\documentclass[11pt]{article}
\usepackage{amsmath,wasysym}
\usepackage{amsthm}
\usepackage{amssymb,amsfonts}
\usepackage{hyperref}
\usepackage{latexsym}
\usepackage{todonotes}
\usepackage{nicefrac}
\usepackage{epsfig}
\usepackage{stmaryrd}
\usepackage{setspace}
\usepackage{tikz-cd}
\usepackage{enumerate}
\usepackage[all]{xypic}
\usepackage{bbm,ifpdf,tikz,mathtools}
\usepackage{tikz-3dplot}

\usetikzlibrary{decorations.pathreplacing} 

\ifpdf
\usepackage{pdfsync}
\fi
\usetikzlibrary{positioning,matrix,arrows,decorations.pathmorphing}
\usepackage{enumitem}

\usetikzlibrary{positioning,arrows}
\usetikzlibrary{decorations.markings}
\usepackage{caption, subcaption}
\usetikzlibrary{arrows.meta,
                bbox}

\newtheorem{theorem}{Theorem}[section]

\newtheorem{lemma}[theorem]{Lemma}
\newtheorem{proposition}[theorem]{Proposition}

\newtheorem{conjecture}[theorem]{Conjecture}
\newtheorem*{main-result}{Main-Result}
{
\theoremstyle{definition}

\newtheorem{example}[theorem]{Example}

\newtheorem{remark}[theorem]{Remark}

}

\newcommand{\im}{\operatorname{im}}
\newcommand{\id}{\operatorname{id}}

\renewcommand{\dim}{\operatorname{dim}}

\newcommand{\excise}[1]{}

\newcommand{\cK}{\mathcal{K}}
\newcommand{\bigmid}{\;\Big{|}\;}

\newcommand{\Z}{{\mathbb{Z}}}
\newcommand{\Q}{{\mathbb{Q}}}
\newcommand{\R}{{\mathbb{R}}}

\newcommand{\Hom}{\operatorname{Hom}}
\newcommand{\cZ}{\mathcal{Z}}
\newcommand{\cI}{\mathcal{I}}
\newcommand{\cJ}{\mathcal{J}}

\newcommand{\la}{\lambda}
\renewcommand{\a}{\alpha}

\newcommand{\eps}{\epsilon}

\newcommand{\sat}{\operatorname{sat}}

\renewcommand{\(}{\left(}
\renewcommand{\)}{\right)}

\newcommand{\init}{\operatorname{in}}

\newcommand{\T}{\mathbb{T}}

\newcommand{\cO}{\mathcal{O}}

\renewcommand{\and}{\qquad\text{and}\qquad}

\newcommand{\Aut}{\operatorname{Aut}}

\newcommand{\Conf}{\operatorname{Conf}}

\newcommand{\Sym}{\operatorname{Sym}}
\newcommand{\OT}{\operatorname{OT}}

\newcommand{\IZ}{\operatorname{IZ}}

\newcommand{\Rees}{\operatorname{Rees}}
\newcommand{\rk}{\operatorname{rk}}
\newcommand{\gr}{\operatorname{gr}}

\newcommand{\Supp}{\operatorname{Supp}}
\newcommand{\bOT}{\overline{\OT}}

\usepackage{todonotes}

\begin{document}
\spacing{1.2}
\noindent{\Large\bf The geometry of zonotopal algebras I:\\ cohomology of graphical configuration spaces
}\\

\noindent{\bf Colin Crowley}\footnote{Supported by NSF grants DMS-2039316 
and DMS-1926686.}\\
Department of Mathematics, University of Oregon, Eugene, OR\\
and School of Mathematics, Institute for Advanced Study, Princeton, NJ
\vspace{.1in}

\noindent{\bf Galen Dorpalen-Barry}\footnote{Supported by NSF grant DMS-2039316.}\\
Department of Mathematics, Texas A\&M University, College Station, TX
\vspace{.1in}

\noindent{\bf Andr\'e Henriques}\\
Mathematical Institute, University of Oxford, Oxford, UK
\vspace{.1in}

\noindent{\bf Nicholas Proudfoot}\footnote{Supported by NSF grants DMS-1954050, DMS-2053243, and DMS-2344861.}\\
Department of Mathematics, University of Oregon, Eugene, OR\\

{\small
\begin{quote}
\noindent {\em Abstract.} 
Zonotopal algebras of vector arrangements are combinatorially-defined 
algebras with connections to approximation theory,
introduced by Holtz and Ron and independently by Ardila and Postnikov.
We show that the internal zonotopal algebra of a cographical vector 
arrangement is isomorphic to the cohomology ring of a certain 
configuration space introduced by Moseley, Proudfoot, and Young.  We also 
study an integral form of this algebra,
which in the cographical case is isomorphic to the integral cohomology ring.
Our results rely on interpreting the internal zonotopal 
algebra of a totally unimodular arrangement as an orbit harmonics ring, 
that is, as the associated graded of the ring of functions on a finite set of lattice points.
\end{quote} }

\section{Introduction}
Let $\Lambda$ be a finite rank lattice, $A$ a finite set, and $\chi:A\to\Lambda$ a map of sets whose image spans $\Lambda$.
We refer to such a structure as a {\bf vector arrangement}.  
We will assume that $\chi$ is {\bf totally unimodular}, meaning that,
for any subset $S\subset A$, the sublattice of $\Lambda$ spanned by $\chi(S)$ is saturated.
The polytope $$Z(\chi) := \left\{\sum_{a\in A} c_a \chi(a)\bigmid 0\leq c_a \leq 1\right\}\subset\Lambda_\R$$
is called the {\bf zonotope} of $\chi$.  This can be regarded as the Minkowski sum of the line segments defined by the elements $\chi(a)$,
or as the image of the unit cube under the map from $\R^A$ to $\Lambda_\R$ determined by $\chi$.
We will be particularly interested in the finite set $\cZ_-(\chi)$ of lattice points in the interior of $Z(\chi)$.

The {\bf internal zonotopal algebra} $\IZ(\chi;\Q)$ is an algebra whose
dimension is equal to the cardinality of $\cZ_-(\chi)$.
It was independently defined by Ardila and Postnikov \cite{ArPo} from a purely combinatorial perspective,
and by Holtz and Ron \cite{Zonotopal}, who were motivated by approximation theory.\footnote{They also
define an external zonotopal algebra and a central zonotopal algebra, whose dimensions are the number of
lattice points and the volume of $Z(A,\chi)$, respectively.  These algebras will not be addressed in this paper.}
Further study of these algebras appears in \cite{Lenz,ArPo-correction,Kirillov-Nenashev,Berget-reflection,RRT,Brodsky}.
The formal definition is as follows:
for any $\a\in\Lambda^*_\Q = \Hom(\Lambda,\Q)$, let $d(\a)$ be the number of elements $a\in A$ such that $\langle \a,\chi(a)\rangle \neq 0$.
Then
\begin{equation}\label{izdef}\IZ(\chi;\Q) := \Sym \Lambda^*_\Q \big{/} \langle \a^{d(\a)-1}\mid 0\neq \a \in\Lambda^*\rangle.\end{equation}

Our main result is concerned with the case where $\chi$ is the cographical vector arrangement associated with a graph $\Gamma$.
In this case, we prove that $\IZ(\chi;\Q)$ is isomorphic to the cohomology ring of a certain configuration space introduced by 
Moseley, Proudfoot, and Young \cite{MPY} (Theorem \ref{cohomology}).  
This result is the first step in resolving a conjecture from that paper (Conjecture \ref{MPY-conj}).
We also introduce an integral form $\IZ(\chi;\Z)$ of the algebra, and prove
that this coincides with the integral cohomology ring of the configuration space.


\subsection{Orbit harmonics and an integral form}\label{sec:new}
Let $R(\Lambda; \Q) := \Sym\Lambda^*_\Q$ be the graded ring of $\Q$-valued polynomial functions on $\Lambda$, and let 
$R_i(\Lambda,\Q)\subset R(\Lambda; \Q)$ be the subgroup consisting of polynomials of degree at most $i$.
Let $R(\chi; \Q)$ be the ring of $\Q$-valued functions on the finite set $\cZ_-(\chi)$.
This is a boring ring (it is a direct sum of finitely many copies of $\Q$), but it is endowed with an interesting filtration
$$0\subset R_0(\chi;\Q) \subset R_1(\chi;\Q) \subset R_2(\chi;\Q) \subset \cdots \subset R(\chi,\Q),$$
where $R_i(\chi;\Q)$ is the image of $R_i(\Lambda,\Q)$ under the restriction map.
The analogous filtration can be defined on the set of functions of any finite subset of a vector space, and 
the associated graded ring is called the {\bf orbit harmonics ring}; see for example \cite{CMR}.
Geometrically, it is the coordinate ring of the scheme-theoretic limit of the set $t\cdot \cZ_-(\chi)\subset \Lambda_\Q$ as $t$ goes to zero.
The following proposition both motivates the definition of the internal zonotopal algebra and provides an explanation for the fact that its total dimension
is equal to the cardinality of $\cZ_-(\chi)$.

\begin{proposition}\label{Q-version}
The internal zonotopal algebra $\IZ(\chi;\Q)$ is isomorphic to the orbit harmonics ring 
$$\gr R(\chi;\Q) := \bigoplus_{i\geq 0} R_i(\chi;\Q)/R_{i-1}(\chi;\Q).$$
\end{proposition}

\begin{remark}\label{RRT graphical}
Proposition \ref{Q-version} is equivalent to \cite[Theorem 5.10]{Zonotopal}, and also to \cite[Proposition 3.5]{RRT} when the vector arrangement is graphical.
However, the three papers use sufficiently different languages that we will provide our own short proof in Section \ref{sec:Q}.
\end{remark}

\begin{remark}
Lenz \cite[Theorem 2]{Lenz} provides an isomorphism between the linear dual of $\IZ(\chi;\Q)$ and $R(\chi;\Q)$
that is unrelated to Proposition \ref{Q-version}.
Lenz considers the {\bf box spline} $M_\chi$, a piecewise polynomial function on $\Lambda_\R$
supported on $Z(\chi)$, whose value at a point $z\in Z(\chi)$ is proportional to the volume of the fiber over $z$ of the natural map $[0,1]^A\to Z(\chi)$.
He identifies the dual of $\IZ(\chi;\Q)$ with a space of differential operators on $R(\Lambda;\Q)$, applies these differential operators to $M_\chi$,
and shows that the restrictions of the resulting functions to $\cZ_-(\chi)$ form a basis for $R(\chi;\Q)$.  This statement was originally
conjectured by Holtz and Ron \cite[Conjecture 1.8]{Zonotopal}, and was one of their primary motivations for introducing the algebra.
\end{remark}

One consequence of Proposition \ref{Q-version} is that it suggests a canonical integral structure on the internal zonotopal algebra.
Let $R(\chi;\Z)$ be the graded ring of $\Z$-valued functions on $\cZ_-(\chi)$, and let $R_i(\chi;\Z) := R_i(\chi;\Q)\cap R(\chi;\Z)$.
Motivated by Proposition \ref{Q-version}, we define
$$\IZ(\chi;\Z) := \gr R(\chi;\Z) = \bigoplus_{i\geq 0} R_i(\chi;\Z)/R_{i-1}(\chi;\Z),$$
which has the property that $\IZ(\chi;\Z)\otimes\Q\cong \IZ(\chi;\Q)$.

\begin{remark}
Why this integral form, and not some other one?  For example, we could have defined $\IZ(\chi;\Z)$ to be the quotient
of $\Sym \Lambda^*$ by the ideal $\langle \a^{d(\a)-1}\mid 0\neq \a\in\Lambda^*\rangle$, which is not isomorphic
to the ring we have defined (see Example \ref{ex:cycle}).
Our main motivation is the cohomological interpretation that will appear in Theorem \ref{cohomology}
for cographical vector arrangements.
\end{remark}

The ring $R(\chi;\Z)$ is a {\bf filtered binomial ring}, meaning that for any function $\eta\in R_i(\chi;\Z)$ and any natural number $m$,
it makes sense to talk about the binomial coefficient $\binom{\eta}{m}\in R_{mi}(\chi;\Z)$.\footnote{A binomial ring is the same as a $\la$-ring
for which every Adams operations is the identity \cite[Theorem 5.3]{Yau}.}  This induces on $\IZ(\chi;\Z)$ the structure of a {\bf divided powers ring}:
for each $m\geq 0$ and each $e\in \IZ^i(\chi;\Z)$ with $i>0$, we have a distinguished element $e^{[m]}\in \IZ^{mi}(\chi;\Z)$ such that $$m!\, e^{[m]} = e^m.$$
Concretely, the element $e^{[m]}$ is obtained by lifting $e$ to $\eta\in R_i(\chi;\Z)$, taking a binomial coefficient $\binom{\eta}{m}\in R_{mi}(\chi;\Z)$,
and projecting back to $\IZ^{mi}(\chi;\Z)$.  
Since $R_i(\chi;\Z)$ is by definition saturated, $\IZ(\chi;\Z)$ is torsion-free, as an abelian group, thus its divided powers structure is unique.

A divided powers algebra is almost never generated as a ring in degree 1.
The following proposition, which we prove in Section \ref{sec:filtrations}, 
says that it is generated as a ring by classes of the form $e^{[m]}$ for $e\in \IZ^1(\chi)$ and $m\geq 0$.

\begin{proposition}\label{Z-version}
The ring $\IZ(\chi;\Z)$ is generated in degree $1$ as a divided powers ring.
\end{proposition}

\begin{example}\label{ex:cycle}
Suppose that $\Lambda = \Z$, $A = \{1,\ldots,k\}$, and $\chi:A\to\Lambda$ takes every element of $A$ to the element $1\in\Lambda$.
Then $Z(\chi) = [0,k]$ and $\cZ_-(\chi) = \{1,2,\ldots,k-1\}$.
There exists an element $e\in \IZ^1(\chi;\Z)$ such that $\IZ(\chi;\Z) = \Z\{1,e,e^{[2]},\ldots,e^{[k-2]}\}$ is the truncation to degree $k-2$ of the free divided powers ring in one variable.
Note that $\IZ(\chi;\Z)$ is not generated {\em as a ring} by the element $e$; for example, the class $e^{[2]} = e^2/2$ is not contained in the subring generated by $e$.
\end{example}

\subsection{Graphical configuration spaces}
Let $\Gamma = (V,A,h,t)$ be a directed graph: $V$ is the set of vertices, $A$ is the set of arrows,
and $h,t:A\to V$ are the head and tail maps.  Let $\Lambda = H^1(\Gamma;\Z)$.  For each arrow $a\in A$, we have a map to the (oriented) circle
given by contracting all other arrows, and pulling back the generator of $H^1(S^1;\Z)$ along this map defines a cohomology class 
$\chi_\Gamma^!(a)\in\Lambda$.  We refer to $\chi_\Gamma^!$ as a {\bf cographical vector arrangement}.
Its associated matroid is a cographical matroid, and it is Gale dual to the vector arrangement typically associated with a graph
(see Section \ref{sec:OT}).
Let $$\cZ_-(\Gamma) := \cZ_-(\chi_\Gamma^!),\qquad
R(\Gamma;\Z) := R(\chi_\Gamma^!;\Z)\and
\IZ(\Gamma;\Z) := \IZ(\chi_\Gamma^!;\Z).$$

Let $G$ be a Lie group, and  
let $\Conf(\Gamma,G)$ denote the space of maps from the vertex set into $G$
with the property that adjacent vertices map to distinct elements of $G$.  Let 
$$X(G,\Gamma) := \Conf(\Gamma,G)/G^{\pi_0(\Gamma)},$$
where $G^{\pi_0(\Gamma)}$ acts by left translation on the various connected components.
Note that the space $X(G,\Gamma)$ does not depend on the orientations of the edges.

\begin{example}
Let $\Gamma = K_{n}$ be the complete graph on the vertex set $[n]$, and let $G = SU(2)$.
Using the fact that $SU(2)\setminus\{\id\}\cong\R^3$, we may translate the last point to the identity
and thus identify $X(G,\Gamma)$ with the configuration space
of $n-1$ distinct labeled points in $\R^3$.  A feature of this construction is that it demonstrates that the natural action of $S_{n-1}$
on the configuration space of $n-1$ points in $\R^3$ extends in a canonical way to an action of $S_{n}$.
When $\Gamma$ is does not have a vertex adjacent to every other vertex (a cone vertex), this space is much harder to analyze.
\end{example}

\begin{remark}
In Section \ref{sec:homotopy}, we make use of the fact that $X(G,\Gamma)$ is diffeomorphic to the space $Y(G,\Gamma)$ of balanced $G$-gain graphs with underlying graph $\Gamma$ and all coefficients nontrivial.
\end{remark}

The following theorem is our main result; the proof appears in Section \ref{sec:proof}.

\begin{theorem}\label{cohomology}
There is a canonical degree-halving isomorphism
$$H^*(X(SU(2),\Gamma);\Z) \cong \IZ(\Gamma;\Z).$$
Moreover, if $\T := U(1)/\{\pm 1\}$ acts on $X(SU(2),\Gamma)$ by inverse right multiplication,
then we have a canonical degree-preserving isomorphism
$$H^*_\T(X(SU(2),\Gamma);\Z) \cong 
\Rees R(\Gamma;\Z) := \bigoplus_{i\geq 0}^\infty u^i R_i(\Gamma;\Z) \subset R(\Gamma;\Z[u])$$
of graded algebras over $\Z[u]\cong H^*_\T(*;\Z)$ (with $\deg u = 2$).
\end{theorem}

\begin{remark}
This is one of many examples in the literature where an orbit harmonics ring is isomorphic
to the cohomology ring of a space; see e.g. \cite{GM,kosdef,Pawlowski-Rhoades,CMR,GLW}. 
The distinctly new feature of Theorem \ref{cohomology} is that the orbit harmonics interpretation holds over the integers,
thus exhibiting a divided powers structure on the integral cohomology ring.
\end{remark}

\begin{remark}\label{automorphisms}
All of the isomorphisms in Theorem \ref{cohomology} are canonical and therefore equivariant with respect to the action of the group
of automorphisms of $\Gamma$.  The representation category of a finite group over the rational numbers is semisimple,
which means that the representation
$$H^*(X(SU(2),\Gamma);\Q) \cong \IZ(\Gamma;\Q) \cong \gr R(\Gamma;\Q)$$
becomes isomorphic, after forgetting the grading, to the permutation representation $R(\Gamma;\Q)$ with basis given by the finite set $\cZ_-(\Gamma)$.
\end{remark}

\begin{remark}
Consider the special case where $\Gamma = C_{k}$ is the cycle of length $k$ for some even number $k$.
In this case, we can define a natural inclusion from $X(SU(2),C_{k})$ to
the based loop space $\Omega SU(2) := LSU(2)/SU(2)$.\footnote{The algebro-geometric version of $\Omega SU(2)$
is known as the affine Grassmannian for $SU(2)$, and is homotopy equivalent to $\Omega SU(2)$.}
Given a function $\kappa:\{1,\ldots,k\}\to SU(2)$ defining an element of $X(SU(2),C_{k})$, 
we consider the piecewise geodesic loop that starts at $\kappa(1)$, takes the minimal geodesic
in $SU(2)\cong S^3$ to $-\kappa(2)$, then to $\kappa(3)$, then $-\kappa(4)$, and so on, eventually returning to $\kappa(1)$.  
The fact that $\kappa(i)\neq \kappa(i+1)$ guarantees
that $-\kappa(i+1)$ is not the antipode of $\kappa(i)$, so the minimal geodesic is well defined.
The cographical vector arrangement $\chi_{C_k}^!$ is the one appearing in Example \ref{ex:cycle},
so Theorem \ref{cohomology} says that the cohomology of $X(SU(2),C_{k})$ is as described in that example: the
truncation to degree $k-2$ of the free divided powers ring in one variable. 
The cohomology ring of $\Omega SU(2)$ is isomorphic to the free divided powers ring in one variable \cite{HHH},
and the restriction map on cohomology is the truncation map.

When $k$ is odd, the analogous map lands not in $\Omega SU(2)$, but rather in $\Omega SO(3)$, which is homeomorphic
to a disjoint union of two copies of $\Omega SU(2)$.  More generally, for an arbitrary graph $\Gamma$, we obtain
an inclusion from $X(SU(2),\Gamma)$ to the space of maps from the geometric realization of $\Gamma$ to $SO(3)$, modulo
the action of $SO(3)^{\pi_0(\Gamma)}$.  Each connected component of this space is homotopy equivalent to 
the $g^\text{th}$ power of $\Omega SU(2)$, where $g = \dim H^1(\Gamma;\Q)$.
This provides additional insight into why the cohomolology ring of $X(SU(2),\Gamma)$ should carry the structure of a divided powers ring.
\end{remark}

\subsection{The reduced Orlik--Terao algebra}\label{sec:OT}
We conclude the introduction by recording one of our primary motivations for these results, which is the subject
of a companion paper \cite{ZA2}.  
Assume that $\chi(a)\neq 0$ for all $a\in A$, and 
let $\OT(\chi;\Q)$ be the subring of rational functions on $\Lambda^*_\Q$ generated by the functions $\{\chi(a)^{-1}\mid a\in A\}$.
This is known as the {\bf Orlik--Terao algebra} of $\chi$.
There is a natural map from $\Lambda^*_\Q$ to $\OT(\chi)$ taking $\a$ to $\sum_{a} \langle\a,\chi(a)\rangle \chi(a)^{-1}$, and the quotient 
$\bOT(\chi)$ of $\OT(\chi)$ by the ideal generated by the image of this map is called the {\bf reduced Orlik--Terao algebra}.

Let $\Gamma$ be a directed graph, and let $\chi_\Gamma:A\to \Z^A/H_1(\Gamma;\Z)$ be the {\bf graphical vector arrangement}.
This arrangement is Gale dual to the cographical vector arrangement $\chi_\Gamma^!$, and in particular the two associated matroids are dual to each other.
Let 
$\OT(\Gamma) := \OT(\chi_\Gamma)$ and $\bOT(\Gamma) := \bOT(\chi_\Gamma)$.
From the construction, is evident that $\bOT(\Gamma)$ is a representation of the group of automorphisms of $\Gamma$ as a directed graph.
With some care, one can show that in fact the group of automorphisms of the underlying undirected graph also acts on $\bOT(\Gamma)$.
We denote the larger group by $\Aut(\Gamma)$; this group also acts on $H^*(X(SU(2),\Gamma);\Q))$, as noted in Remark \ref{automorphisms}.

\begin{conjecture}\label{MPY-conj}{\em \cite[Conjecture 2.16]{MPY}}
For any graph $\Gamma$, there exists a degree-halving isomorphism of graded $\Aut(\Gamma)$-representations
$$H^*(X(SU(2),\Gamma);\Q)) \cong \bOT(\Gamma).$$
\end{conjecture}

We note that the representations in 
Conjecture~\ref{MPY-conj} are typically not isomorphic as rings, only as graded representations.
In the case where $\Gamma = K_n$ and $\Aut(\Gamma) = S_n$, this conjecture was proved by Pagaria \cite{Pagaria}.
He did not produce a canonical isomorphism, but rather proved the result by studying the generating functions
for the symmetric functions obtained by applying the Frobenius characteristic map to the graded pieces of the representations.
In \cite{ZA2}, we prove that, for any Gale dual pair of vector arrangements $\chi:A\to \Lambda$ and $\chi^!:A\to\Lambda^!$, 
the reduced Orlik--Terao algebra $\bOT(\chi)$
may be canonically identified with the linear dual of $\IZ(\chi^!;\Q)$.  Since finite dimensional representations of finite groups over $\Q$ are self-dual,
this result, combined with Theorem \ref{cohomology} and applied to the pair $\chi_\Gamma$ and $\chi_\Gamma^!$, implies Conjecture \ref{MPY-conj}.

\vspace{\baselineskip}
\noindent
{\em Acknowledgments:}
The authors are grateful to Federico Ardila, Andy Berget, Christin 
Bibby, Vic Reiner, Brendon Rhoades, and Benjamin Schr\"oter for valuable conversations. We thank Ethan 
Partida for helpful discussions about orbit harmonics and the connection of 
our work to \cite{RRT}. The first author thanks the Simons 
foundation for support.  The fourth author thanks All Souls College for 
its hospitality during the preparation of this manuscript.

\section{Zonotopal Algebra}
In this section, we begin by proving Proposition \ref{Q-version}, and then proceed to give explicit descriptions of the rings $R(\chi;\Z)$
and $\IZ(\chi;\Z)$.

\subsection{Algebras over \boldmath{$\Q$}}\label{sec:Q}
For any element $\a\in\Lambda^*$, we define $$\Supp(\a) := \{a\in A\mid \langle \a, \chi(a) \rangle \neq 0\}.$$
A nonzero primitive element of minimal support is called a {\bf cocircuit}.  
The unimodularity hypothesis implies that, for any cocircuit $\a$, we
have $\langle \a, \chi(a)\rangle \in \{-1,0,1\}$ for all $a\in A$; see for example \cite[Claim 5.42]{tutte} or \cite[Theorem 21.1]{schrijver}.  
We will write $d_\pm(\a) := |\{a\in A\mid \langle \a, \chi(a)\rangle =\pm 1\}|$, so that $d(\a) = d_+(\a) + d_-(\a) = |\Supp(\a)|$.  
By \cite[Corollary 7.17]{Z}, cocircuits are in bijection with pairs of parallel facets of the zonotope $Z(\chi)$:
\begin{equation}\label{facets}
Z(\chi) = \left\{z\in \Lambda_\R\mid \text{{$-d_-(\a) \leq \langle \a, z\rangle \leq d_+(\a)$ for all cocircuits $\a$}}\right\}.
\end{equation}
Following \cite[Section 5.1]{Zonotopal},
we define the ideal $$\cI_-(\chi;\Q) := \left\langle \a^{d(\a)-1} \bigmid \text{$\a$ a cocircuit}\right\rangle \subset R(\Lambda;\Q).$$
This is {\em a priori} only contained in the ideal $\left\langle \a^{d(\a)-1} \mid 0\neq \a\in\Lambda^*\right\rangle$ appearing in
Equation \eqref{izdef}, but in fact that containment is an equality \cite[Lemma 1]{ArPo-correction}.
Thus we have $$\IZ(\chi;\Q) = R(\Lambda;\Q)/\cI_-(\chi;\Q).$$

For any (inhomogeneous) polynomial $f\in R(\Lambda;\Q)$, let $\init(f)$ denote the sum of its top degree terms. 
For any ideal $\cJ \subset R(\Lambda;\Q)$, let $\init\cJ := \{\init(f)\mid f\in \cJ\}$.
If $\cJ$ is not homogeneous, then $R(\Lambda;\Q)/\cJ$ is filtered rather than graded, with 
the $i^\text{th}$ filtered piece consisting of the image of $R_i(\Lambda;\Q)$, and there is a canonical isomorphism
$$\gr\(R(\Lambda;\Q)/\cJ\) \cong R(\Lambda;\Q)/\init \cJ.$$

For any cocircuit $\a$, Equation \eqref{facets} tells us that $\a$
takes the values $[-d_-(\a), d_+(\a)]$ on $Z(A,\chi)$, and therefore the values $(-d_-(\a), d_+(\a)) \cap \Z$ on $\cZ_-(\chi)$.  It follows that 
the binomial coefficient $\binom{\a+d_-(\a)-1}{d(\a)-1}$ vanishes on $\cZ_-(\chi)$.
Consider the ideal 
$$\cK_-(\chi;\Q) := \left\langle \binom{\a+d_-(\a)-1}{d(\a)-1}\bigmid \text{$\a$ a cocircuit}\right\rangle\subset R(\Lambda;\Q),$$
and let
$$\tau(\Q):R(\Lambda;\Q)/\cK_-(\chi;\Q)\to R(\chi;\Q)$$ denote the restriction homomorphism.  
This is a surjective homomorphism of filtered algebras, and the $i^\text{th}$ filtered piece of the target is equal to the image of
the $i^\text{th}$ filtered piece of the source.  It follows that
the induced homomorphism $\gr\tau(\Q)$ of associated graded algebras is also surjective.

\begin{proof}[Proof of Proposition \ref{Q-version}]
The generators of $\cI_-(\chi;\Q)$ are precisely the leading terms of the generators of $\cK_-(\chi;\Q)$, thus $\cI_-(\chi;\Q)\subset\init\cK_-(\chi;\Q).$
This means that we have surjections
$$\IZ(\chi;\Q) = R(\Lambda;\Q)/\cI_-(\chi;\Q) \twoheadrightarrow R(\Lambda;\Q)/\init \cK_-(\chi;\Q) \cong 
\gr\(R(\Lambda;\Q)/\cK_-(\chi;\Q)\)\twoheadrightarrow \gr R(\chi;\Q).$$
By \cite[Proposition 1.1(3)]{Zonotopal}, the total dimension of $\IZ(\chi;\Q)$ is equal to the cardinality of $\cZ_-(\chi)$.  This is also equal to the dimension
of $R(\chi;\Q)$, and therefore of $\gr R(\chi;\Q)$.  Thus both of the surjections above must be isomorphisms.
\end{proof}

\begin{remark}
We have just proved that $\gr\tau(\Q)$ is an isomorphism,
which implies that $\tau(\Q)$ is an isomorphism, as well.
That is, the kernel of the restriction homomorphism from $R(\Lambda;\Q)$ to $R(\chi;\Q)$ is equal to $\cK_-(\chi;\Q)$.
\end{remark}

\subsection{Integer valued functions and divided powers}\label{sec:dp}
Consider the ring $R(\Lambda;\Z)\subset R(\Lambda;\Q)$ of integer valued polynomial functions on $\Lambda$, and let
$$R_i(\Lambda;\Z) := R_i(\Lambda;\Q) \cap R(\Lambda;\Z).$$
This defines a filtration, and we define
$$S(\Lambda;\Z) := \gr R(\Lambda;\Z) = \bigoplus_{i=0}^\infty R_i(\Lambda,\Z)/ R_{i-1}(\Lambda,\Z).$$
Concretely, $R_i(\Lambda,\Z)$ is spanned by products of binomial coefficients of the form
$$\binom{\a_1}{i_1}\cdots \binom{\a_k}{i_k}$$
for $\a_1,\ldots,\a_k\in\Lambda^*$ and $i_1,\ldots,i_k\in\mathbb{N}$ with $i_1+\cdots+i_k\leq i$ \cite[Theorem 5.28]{Yau}.
Passing to the associated graded, we may identify 
$S(\Lambda;\Z)$ with the free divided powers algebra on $\Lambda^*$, that is, the subgroup of $R(\Lambda;\Q)$ spanned by elements
of the form
$$\frac{\a_1^{i_1}}{i_1!}\cdots \frac{\a_k^{i_k}}{i_k!}.$$  Note that $R(\Lambda;\Z)$ and $S(\Lambda;\Z)$ are both subrings of $R(\Lambda;\Q)$, and we have
$$R(\Lambda;\Z)\otimes\Q = R(\Lambda;\Q) = S(\Lambda;\Z)\otimes\Q.$$

\excise{
If the rank of $\Lambda$ is positive, the rings $R_\Lambda$ and $S_\Lambda$ do not 
have any finitely generated proper ideals whose quotient is finitely generated as an abelian group.
For this reason, when defining ideals in these algebras, we will often saturate them as abelian groups.
That is, we define $$\langle f_1,\ldots, f_r\rangle^{\sat} := \{f\mid \text{there exists $n>0$ such that $nf\in \langle f_1,\ldots, f_r$}\rangle\}.$$
}


For any polynomial $f\in R(\Lambda;\Z)$, we have $\init(f)\in S(\Lambda;\Z)$.
As in the previous section, for any ideal $\cJ\subset R(\Lambda;\Z)$, we define
$$\init\cJ := \{\init(f)\mid f\in \cJ\}\subset S(\Lambda;\Z),$$
and we identify $\gr\(R(\Lambda;\Z)/\cJ\)$ with $S(\Lambda;\Z)/\init\cJ$.

\excise{
\begin{remark}\label{warning}
It is not true in general that $\init(\cJ)^{\sat} = \init\(\cJ^{\sat}\)$.  For example, if $\Lambda = \Z$
and $\cJ = \langle 2x+1\rangle \subset R(\Lambda;\Z)$, then the saturation of $\init(\cJ)$ contains $x$,
but the initial ideal of the saturation of $\cJ$ does not.
\end{remark}
}

\subsection{Two filtrations}\label{sec:filtrations}
There are two natural ways to define a filtration of $R(\chi;\Z)$.  The first, which appeared in the introduction,
is to define $$R_i(\chi;\Z) := R_i(\chi;\Q) \cap R(\chi;\Z).$$
This definition has the desirable property that the ring
$$\IZ(\chi;\Z) := \bigoplus_{i\geq 0} R_i(\chi;\Z)/R_{i-1}(\chi;\Z)$$ is clearly torsion-free as an abelian group.
The second is to define $\tilde R_i(\chi;\Z)$ to be the image of $R_i(\Lambda;\Z)$ in $R(\chi;\Z)$ under the restriction map
from functions on $\Lambda$ to functions on $\cZ_-(\chi)$.  This definition has the different desirable property that the ring
$$\bigoplus_{i\geq 0} \tilde R_i(\chi;\Z)/\tilde R_{i-1}(\chi;\Z)$$ is generated in degree 1 as a divided powers ring
(see the proof of Proposition \ref{Z-version} at the end of this section).
It is immediate that $\tilde R_i(\chi;\Z)\subset R_i(\chi;\Z)$
and that $R_i(\chi;\Z)$ is the saturation of $\tilde R_i(\chi;\Z)$.
However, there could {\em a priori} be an integer valued function on $\cZ_-(\chi)$ that extends to a polynomial of degree $i$ on $\Lambda$,
but not to a polynomial of degree $i$ that takes integer values on $\Lambda$.  We will prove that, in fact, this cannot happen.

\begin{proposition}\label{saturation}
For all $i$, we have $R_i(\chi;\Z) = \tilde R_i(\chi;\Z)$.
\end{proposition}

\begin{remark}
Proposition \ref{saturation} is a special property of $\cZ_-(\chi)$
rather than a statement about all finite subsets of $\Lambda$.  
For example, the function $x/2$ on the set $\{0,2\}\subset\Z$ extends to a polynomial of degree 1 on $\Z$,
but not to a polynomial of degree 1 that takes integer values on $\Z$.
\end{remark}

We will prove Proposition \ref{saturation} by induction using the operators of deletion and contraction, which we now define.
Fix an element $a\in A$, and let $\bar A:= A\setminus\{a\}$.  The element $a$ is called a {\bf loop} if $\chi(a) = 0$, and it is called a {\bf coloop} if $\{\chi(s)\mid s\in \bar A\}$ does not span $\Lambda$.
If $a$ is not a coloop, we define the {\bf deletion} $$\chi':= \chi|_{\bar A}:\bar A\to\Lambda.$$
If $a$ is not a loop, we let $\bar\Lambda := \Lambda/\Z\chi(a)$, we denote the projection from $\Lambda$ to $\bar\Lambda$ by $z\mapsto\bar z$, and we define the {\bf contraction} $$\chi'' := \bar\chi|_{\bar A}:\bar A\to\bar\Lambda.$$

\begin{lemma}\label{bijection}{\em \cite[Lemma 17]{Lenz}}
Suppose that $a$ is neither a loop nor a coloop.  Then we have an inclusion $\cZ_-(\chi')\subset\cZ_-(\chi)$,
and the map $z\mapsto \bar z$ restricts to a bijection
$$\cZ_-(\chi)\setminus\cZ_-(\chi')\to\cZ_-(\chi'').$$
\end{lemma}

Consider the short exact sequence
\begin{equation}\label{R-exact}0\longrightarrow R(\chi'';\Z)\overset{\xi}{\longrightarrow} R(\chi;\Z)\overset{\partial}{\longrightarrow}R(\chi';\Z)\longrightarrow 0,\end{equation}
where $$\xi f(z) := f(\bar z)\and \partial g(z) := g(z + \chi(a)) - g(z).$$
It is clear that $\im\xi=\ker\partial$, and the injectivity of $\xi$ and surjectivity of $\partial$ both follow from Lemma \ref{bijection}.
Restricting the degrees of the functions, we obtain a complex
\begin{equation}\label{Ri-exact}0\longrightarrow R_i(\chi'';\Z)\overset{\xi}{\longrightarrow} R_i(\chi;\Z)\overset{\partial}{\longrightarrow}R_{i-1}(\chi';\Z)\longrightarrow 0.\end{equation}
We still know that $\xi$ is injective and $\partial$ surjective, but it is not {\em a priori} clear that \eqref{Ri-exact} is exact in the middle.
Indeed, if $g\in R_i(\chi;\Z)$ and $\partial g = 0$, we know that there exists $f\in R(\chi'';\Z)$ with $\xi f = \partial g$, but 
it is not obvious that $f\in R_i(\chi'';\Z)$.\footnote{If the projection from $\Lambda$ to $\bar\Lambda$ had a linear section
taking $\cZ_-(\chi'')$ to a subset of $\cZ_-(\chi)$, then $f$ would be the pullback of $g$ along this section,
and we could conclude that $f\in R_i(\chi'';\Z)$.  However, such a section does not always exist.  
In particular, the inverse of the bijection in Lemma \ref{bijection} is typically not linear.}
We can further restrict to functions that extend to integer valued functions on the lattice of bounded degree, and we obtain a complex
\begin{equation}\label{Ritilde-exact}0\longrightarrow \tilde R_i(\chi'';\Z)\overset{\xi}{\longrightarrow} \tilde R_i(\chi;\Z)\overset{\partial}{\longrightarrow}\tilde R_{i-1}(\chi';\Z)\longrightarrow 0,\end{equation}
again with $\xi$ injective and $\partial$ surjective.

\begin{lemma}\label{Riq}
The complex \eqref{Ri-exact} is exact.
\end{lemma}

\begin{proof}
We first consider the complex
\begin{equation}\label{Riq-exact}0\longrightarrow R_i(\chi'';\Q)\overset{\xi}{\longrightarrow} R_i(\chi;\Q)\overset{\partial}{\longrightarrow}R_{i-1}(\chi';\Q)\longrightarrow 0\end{equation}
obtained by tensoring \eqref{Ri-exact} with $\Q$.
By \cite[Proposition 1.9]{Zonotopal} or \cite[Proposition 4.15]{ArPo}, we have
\begin{equation}\label{IZ-Tutte}\sum_{i\geq 0}t^i \dim\IZ^i(\chi;\Q) = t^{|A|-\rk \Lambda}T_{\chi}(0,t^{-1}),\end{equation}
where $T_\chi(x,y)$ is the Tutte polynomial of the matroid associated with $\chi$.
By Proposition \ref{Q-version}, this means that $\dim R_i(\chi;\Q)$ is equal to the evaluation at 1 of the 
truncation of $t^{|A|-\rk \Lambda}T_{\chi}(0,t^{-1})$ to degree $i$,
and similarly for $R_i(\chi'';\Q)$ and $R_{i-1}(\chi';\Q)$.  Using the well-known recursion
$$T_{\chi}(x,y) = T_{\chi'}(x,y) + T_{\chi''}(x,y),$$ we can conclude that $\dim R_i(\chi;\Q) = \dim R_i(\chi'';\Q) + \dim R_{i-1}(\chi';\Q)$.
Since we already know that $\xi$ is injective and $\partial$ is surjective, \eqref{Riq-exact} must be exact.

Now suppose that $g\in R_i(\chi;\Z)$ and $\partial g = 0$.  We have just shown that there exists $f\in R_i(\chi'';\Q)$ such that $\xi f = g$.
Since the map $z\mapsto \bar z$ from $\cZ_-(\chi)$ to $\cZ_-(\chi'')$ is surjective and $g$ takes integer values, $f$ must also take
integer values.  Thus $f\in R_i(\chi'';\Z)$.
\end{proof}

\excise{
\begin{lemma}\label{Ritilde}
If $\tilde R_i(\chi'') = R_i(\chi'')$, then the complex \eqref{Ritilde-exact} is exact.
\end{lemma}

\begin{proof}
Suppose that $g\in \tilde R_i(\chi;\Z)$ and $\partial g = 0$.  By Lemma \ref{Ri}, there exists $f\in R_i(\chi'';\Z)$ such that $\xi f = g$.
The lemma then follows from the assumption that $\tilde R_i(\chi'') = R_i(\chi'')$.
\end{proof}
}

\begin{proof}[Proof of Proposition \ref{saturation}]
If we have any coloops, then $\cZ_-(\chi)$ is empty and the proposition is trivial.
If all elements are loops, then $\Lambda = 0$, $\cZ_-(\chi) = \Lambda$, and again the proposition is trivial.
Thus we may assume that there exists an element $a\in A$ that is neither a loop nor a coloop.  We may also assume by induction 
on the cardinality of $A$ that the statement holds for both the deletion $\chi'$ and the contraction $\chi''$.

We have a commutative diagram of the following form:
\[
\begin{tikzcd}
0 \ar[rr] && \tilde R_i(\chi'';\Z) \ar[rr, "\xi"]\ar[dd, "="] && \tilde R_i(\chi;\Z) \ar[rr, "\partial"]\ar[dd, hook] && \tilde R_{i-1}(\chi';\Z) \ar[rr]\ar[dd, "="] && 0\\ \\ 
0\ar[rr] &&R_i(\chi'';\Z) \ar[rr, "\xi"] && R_i(\chi;\Z) \ar[rr, "\partial"] && R_{i-1}(\chi';\Z) \ar[rr] && 0.
\end{tikzcd}
\]
Lemma \ref{Riq} tells us that the bottom row is exact, and since we already know that $\partial$ is surjective on the top row,
we can conclude that the top row is exact, as well.
By the Five Lemma, the middle vertical map must be an isomorphism.
\end{proof}

\begin{proof}[Proof of Proposition \ref{Z-version}]
Let $e\in \IZ^i(\chi;\Z)$ be given, and lift $e$ to an element $\eta\in R_i(\chi;\Z) = \tilde R_i(\chi;\Z)$.
By definition of $\tilde R_i(\chi;\Z)$, we may extend $\eta$ to an element $\tilde\eta \in R_i(\Lambda;\Z)$.
Then $\tilde\eta$ can be written as a linear combination of functions of the form
$$\binom{\a_1}{i_1}\cdots \binom{\a_k}{i_k}$$
for $\a_1,\ldots,\a_k\in\Lambda^*$ and $i_1,\ldots,i_k\in\mathbb{N}$ with $i_1+\cdots+i_k\leq i$.
It follows that $e$ lies in the divided powers subring of $\IZ(\chi;\Z)$ generated by the images in $\IZ^1(\chi;\Z)$ of $\a_1,\ldots,\a_k$.
\end{proof}

\subsection{Presentations}\label{sec:Z}
In Section \ref{sec:Q}, we gave explicit presentations of the rings $R(\chi;\Q)$ and $\IZ(\chi;\Q)$ as quotients of $R(\Lambda;\Q)$.
In this section, we give analogous descriptions of $R(\chi;\Z)$ and $\IZ(\chi;\Z)$ as quotients of $R(\Lambda;\Z)$ and $S(\Lambda;\Z)$, respectively.

If the rank of $\Lambda$ is positive, the rings $R_\Lambda$ and $S_\Lambda$ do not 
have any finitely generated proper ideals whose quotient is finitely generated as an abelian group.
For this reason, when defining ideals in these algebras, we often saturate them as abelian groups.
That is, we define $$\langle f_1,\ldots, f_r\rangle^{\sat} := \{f\mid \text{there exists $n>0$ such that $nf\in \langle f_1,\ldots, f_r$}\rangle\}.$$
Note that it is not true in general that $\init(\cJ)^{\sat} = \init(\cJ^{\sat})$.  For example, if $\Lambda = \Z$
and we let $\cJ = \langle 2x+1\rangle \subset R(\Lambda;\Z)$, then the saturation of $\init(\cJ)$ contains $x$,
but the initial ideal of the saturation of $\cJ$ does not.

Consider the ideals
$$\cK_-(\chi;\Z) := \left\langle \binom{\a+d_-(\a)-1}{d(\a)-1}\bigmid \text{$\a$ a cocircuit}\right\rangle^{\sat}\subset\;\; R(\Lambda;\Z)$$
and 
$$\cI_-(\chi;\Z) := \left\langle \frac{\a^{d(\a)-1}}{(d(\a)-1)!}\bigmid \text{$\a$ a cocircuit}\right\rangle^{\sat} \subset\;\; S(\Lambda;\Z).$$

\begin{proposition}\label{Z-presentations}
We have canonical isomorphisms
$$R(\chi;\Z) \cong R(\Lambda;\Z)/\cK_-(\chi;\Z)\and \IZ(\chi;\Z) \cong S(\Lambda;\Z)/\cI_-(\chi;\Z).$$
\end{proposition}

\begin{proof}
We have already argued that the function $\binom{\a+d_+(\a)-1}{d(\a)-1}$ vanishes on $\cZ_-(\chi)$.
Since $R(\chi;\Z)$ is torsion-free as an abelian group, this implies that the saturation of the ideal generated by these functions
vanishes on $\cZ_-(\chi)$, 
thus we have a surjective homomorphism
$$\tau(\Z):R(\Lambda;\Z)/\cK_-(\chi;\Z)\to R(\chi;\Z).$$
The source and target of $\tau(\Z)$ are both free abelian groups, and we proved in Section \ref{sec:Q} that $\tau(\Q) = \tau(\Z)\otimes\Q$
is an isomorphism, thus $\tau(\Z)$ is itself an isomorphism.

The proof of the second statement is similar.  
By Proposition \ref{Z-version}, the natural homomorphism from $S(\Lambda;\Z)$ to $\IZ(\chi;\Z)$ is surjective.
Since $\IZ(\chi;\Z)$ is torsion-free and $\a^{d(\a)-1}$ is contained in the kernel for every cocircuit $\a$, the entire saturated 
ideal $\cI_-(\chi;\Z)$ is contained in the kernel.  This provides us with a 
surjective homomorphism from
$S(\Lambda;\Z)/\cI_-(\chi;\Z)$ to $\IZ(\chi;\Z)$.
The source and target are free abelian groups, and after tensoring with $\Q$, we obtain the isomorphism $\gr\tau(\Q)$ from Section \ref{sec:Q}.
Thus the original homomorphism is an isomorphism.
\end{proof}

\begin{remark}
Note that, since we are saturating anyway, we could have defined 
$$\cI_-(\chi;\Z) := \left\langle \a^{d(\a)-1} \bigmid \text{$\a$ a cocircuit}\right\rangle^{\sat} \subset\;\; S(\Lambda;\Z).$$
However, we emphasize that the saturation step is needed even when we leave the denominators in.  Indeed,
the operation of saturation involves adding infinitely many new generators in arbitrarily high degrees;
we do not know of a more explicit description of these generators.
\end{remark}

\subsection{The cographical case}\label{sec:cographical}
We conclude this section with an explicit description of the objects appearing in Proposition \ref{Z-presentations} for cographical vector arrangements.
Let $\Gamma = (V,A,h,t)$ be a directed graph. 
For any $a\in A$, we will formally write $h(-a):= t(a)$ and $t(-a) := h(a)$.
We define an {\bf oriented cycle} to be a set $C = \{(a_1,\eps_1),\ldots,(a_k,\eps_k)\}\subset A\times\{\pm 1\}$ with the following properties:
\begin{itemize}
\item We have $h(\eps_i a_i) = t(\eps_{i+1} a_{i+1})$ for all $1\leq i\leq k-1$ and $h(\eps_k a_k) = t(\eps_1 a_1)$.
\item The vertices $h(\eps_1 a_1),\ldots,h(\eps_k a_k)$ are distinct.
\item If $k=2$, then the arrows $a_1$ and $a_2$ are distinct.
\end{itemize}
Note that we only regard $C$ as a set, not an ordered $k$-tuple; the ordering is unique up to cyclic permutation.
For any cycle $C$, we define $C^+ := \{a_i\mid \eps_i = 1\}$ and $C^- := \{a_i\mid \eps_i = -1\}$.
Let $$\a_C:= \sum_{i=1}^k \eps_i a_i\in \Lambda^* = H_1(\Gamma;\Z) \subset C_1(\Gamma;\Z).$$
Then $\a_C$ is a cocircuit of $\chi_{\Gamma}^!$ with $d_{\pm}(\a_C) = |C^{\pm}|$, and every cocircuit is of this form.
For any oriented cycle $C = \{(a_1,\eps_1),\ldots,(a_k,\eps_k)\}$, we define the {\bf opposite} cycle
$\bar C := \{(a_1,-\eps_1),\ldots,(a_k,-\eps_k)\}$.  We then have $\a_{\bar C} = - \a_C$ and $d_{\pm}(\a_{\bar C}) = d_{\mp}(\a_C)$.
We define an {\bf oriented \boldmath{$\Theta$}-subgraph} of $\Gamma$ to be a triple $\{C_1,C_2,C_3\}$ of oriented cycles
with $\a_{C_1} + \a_{C_2} + \a_{C_3} = 0$.

The ring $R(\Lambda;\Z)$ is isomorphic to the quotient of the free binomial ring with generators $\a_C$
for each oriented cycle $C$ (in the sense of \cite[Section 5.4]{Yau})
by the linear relations $\a_C + \a_{\bar C}$ for each $C$ and
$\a_{C_1} + \a_{C_2} + \a_{C_3}$ for each
oriented $\Theta$-subgraph $\{C_1,C_2,C_3\}$.
The ring $S(\Lambda;\Z)$ is isomorphic to the quotient of the free divided powers ring with the same generators by the same linear relations.
We have $$\cK_-(\chi^!_\Gamma;\Z) = \left\langle \binom{\a_C+|C^-|-1}{|C|-1}\bigmid \text{$C$ an oriented cycle}\right\rangle^{\sat}\subset\;\;  R(\Lambda;\Z)$$
and $$\cI_-(\chi^!_\Gamma;\Z) = \left\langle \frac{\a_C^{|C|-1}}{(|C|-1)!}\bigmid \text{$C$ an oriented cycle}\right\rangle^{\sat}\subset\;\;  S(\Lambda;\Z).$$
We note that $$\binom{\a_C+|C^-|-1}{|C|-1} = \pm \binom{\a_{\bar C}+|\bar C^-|-1}{|\bar C|-1}\and \frac{\a_C^{|C|-1}}{(|C|-1)!} = \pm \frac{\a_{\bar C}^{|\bar C|-1}}{(|\bar C|-1)!},$$ thus it is enough to choose one orientation of each cycle for the sake of describing these ideals.

\subsection{An example}\label{ex:house-graph}
Let $\Gamma$ be the following graph:
  
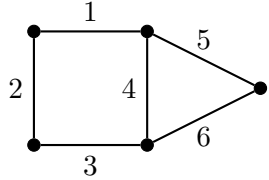
\begin{center}
    \begin{tikzpicture}[scale=.75]
      \node[circle,draw,fill=black,inner sep = .15em] (a) at (0,0) {};
      \node[circle,draw,fill=black,inner sep = .15em] (b) at (0,2) {};
      \node[circle,draw,fill=black,inner sep = .15em] (c) at (2,2) {};
      \node[circle,draw,fill=black,inner sep = .15em] (d) at (2,0) {};
      \node[circle,draw,fill=black,inner sep = .15em] (e) at (4,1) {};
      \draw[-, thick] (b) to node[above]{1} (c);
      \draw[-, thick] (a) to node[left]{2} (b);
      \draw[-, thick] (a) to node[below]{3} (d);
      \draw[-, thick] (c) to node[left]{4} (d);
      \draw[-, thick] (c) to node[above]{5} (e);
      \draw[-, thick] (d) to node[below]{6} (e);
      \end{tikzpicture}
  \end{center}
  We orient the edge 1 from right to left, and choose the remainder of the orientations in such a way so that we have six oriented cycles, namely
  \begin{eqnarray*} C_3 &=& \{(4,1),(5,1),(6,1)\},\\ 
  C_4 &=& \{(1,1),(2,1),(3,1),(4,1)\},\\  
  C_5 &=& \{(1,1),(2,1),(3,1),(6,-1),(5,-1)\},
  \end{eqnarray*}
  and their opposites $\bar C_3$, $\bar C_4$, and $\bar C_5$.
The triple $\{C_3,\bar C_4,C_5\}$ is a
  $\Theta$-subgraph, meaning that 
  $$0 = \a_{C_3} + \a_{\bar C_4} + \a_{C_5} = \a_{C_3} - \a_{C_4} + \a_{C_5}.$$
 
  The lattice $\Lambda = H^2(\Gamma;\Z)$ has rank 2, and we will use the coordinates $\a_{C_4}, \a_{C_3}\in\Lambda^*$ to identify $\Lambda$ with $\Z^2$,
  with $\a_{C_4}$ corresponding to the first (horizontal) coordinate function and $\a_{C_3}$ to the second (vertical) coordinate function.
  Under this identification, we have
  $$\chi(1) = \chi(2) = \chi(3) = (1,0),\qquad \chi(4) = (1,1), 
\and   \chi(5) = \chi(6) = (0,1),$$
and the the zonotope $Z(\Gamma)$ and the set $\cZ_-(\Gamma)$ of internal lattice points are as follows:
  \begin{center}
\excise
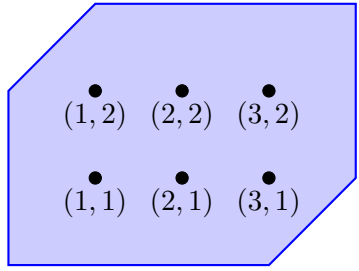
{
        \begin{tikzpicture}[scale=1.5,>=stealth]
          \draw[->] (-0.5,0) -- (1.5,0) node[right] {$f_1$};
          \draw[->] (0,-0.5) -- (0,1.5) node[above] {$f_2$};
          
          \draw[->, very thick, red] (0,0) -- (1,0) node[midway, below] {$(1,0)$};
          \draw[->, very thick, blue] (0,0) -- (1,1) node[midway, right] {$(1,1)$};
          \draw[->, very thick, green!80!black] (0,0) -- (0,1) node[midway, left] {$(0,1)$};
          
      \end{tikzpicture}
      \qquad
}
      \begin{tikzpicture}[scale=1.15]
        \fill[blue!20, opacity=0.5] (0,0) -- (0,2) -- (1,3) -- (4,3) -- (4,1) -- (3,0) -- cycle;
  
        \filldraw[black] (1,1) circle (2pt) node[below] {$(1,1)$};
        \filldraw[black] (1,2) circle (2pt) node[below] {$(1,2)$};
        \filldraw[black] (2,1) circle (2pt) node[below] {$(2,1)$};
        \filldraw[black] (2,2) circle (2pt) node[below] {$(2,2)$};
        \filldraw[black] (3,1) circle (2pt) node[below] {$(3,1)$};
        \filldraw[black] (3,2) circle (2pt) node[below] {$(3,2)$};

        \draw[thick, blue] (0,0) -- (0,2) -- (1,3) -- (4,3) -- (4,1) -- (3,0) -- cycle;

\excise{
        \filldraw[black] (0,0) circle (2pt) node[left] {$(0,0)$};
        \filldraw[black] (3,0) circle (2pt) node[below] {$(3,0)$};
        \filldraw[black] (0,2) circle (2pt) node[left] {$(0,2)$};
        \filldraw[black] (1,3) circle (2pt) node[above] {$(1,3)$};
        \filldraw[black] (4,3) circle (2pt) node[above] {$(4,3)$};
        \filldraw[black] (4,1) circle (2pt) node[right] {$(4,1)$};
}
    \end{tikzpicture}
  \end{center}
The ring $R(\Lambda;\Z)$ is the free binomial ring on the generators $\a_{C_4}$ and $\a_{C_3}$, while $S(\Lambda;\Z)$ is the free divided
powers ring on the same generators.
Following Proposition \ref{Z-presentations}, we have
$$\cK_-(\chi^!_\Gamma;\Z) = \left\langle \binom{\a_{C_3}-1}{2}, \binom{\a_{C_4}-1}{3},\binom{\a_{C_5}+1}{4} \right\rangle^{\sat}$$
and
$$\cI_-(\chi^!_\Gamma;\Z) = \left\langle \frac{\a_{C_3}^{2}}{2!}, \frac{\a_{C_4}^{3}}{3!}, \frac{\a_{C_5}^{4}}{4!}\right\rangle^{\sat}.$$
Below we provide pictorial representations of the functions $\binom{\a_{C_5}+1}{i}$ for $i\in \{0,1,2,3\}$:
  \begin{center}
    \begin{tikzpicture}[scale=.5]
      \draw[] (-1,2) node {$\!\!\!\!\!\!\!\!\!\!\binom{\a_5+1}{0} = $};

      \fill[blue!20, opacity=0.5] (0,0) -- (0,2) -- (1,3) -- (4,3) -- (4,1) -- (3,0) -- cycle;

        \filldraw[black] (1,1) circle (2pt) node[above] {$1$};
        \filldraw[black] (1,2) circle (2pt) node[above] {$1$};
        \filldraw[black] (2,1) circle (2pt) node[above] {$1$};
        \filldraw[black] (2,2) circle (2pt) node[above] {$1$};
        \filldraw[black] (3,1) circle (2pt) node[above] {$1$};
        \filldraw[black] (3,2) circle (2pt) node[above] {$1$};

    

    \end{tikzpicture}
    \hfill
      \begin{tikzpicture}[scale=.5]
        \draw[] (-1,2) node {$\!\!\!\!\!\!\!\!\!\!\binom{\a_5+1}{1} = $};

        \fill[blue!20, opacity=0.5] (0,0) -- (0,2) -- (1,3) -- (4,3) -- (4,1) -- (3,0) -- cycle;

          \filldraw[black] (1,1) circle (2pt) node[above] {$1$};
          \filldraw[black] (1,2) circle (2pt) node[above] {$0$};
          \filldraw[black] (2,1) circle (2pt) node[above] {$2$};
          \filldraw[black] (2,2) circle (2pt) node[above] {$1$};
          \filldraw[black] (3,1) circle (2pt) node[above] {$3$};
          \filldraw[black] (3,2) circle (2pt) node[above] {$2$};
  
      

      \end{tikzpicture}
      \hfill
      \begin{tikzpicture}[scale=.5]
        \draw[] (-1,2) node {$\!\!\!\!\!\!\!\!\!\!\binom{\a_5+1}{2} = $};

        \fill[blue!20, opacity=0.5] (0,0) -- (0,2) -- (1,3) -- (4,3) -- (4,1) -- (3,0) -- cycle;

          \filldraw[black] (1,1) circle (2pt) node[above] {$0$};
          \filldraw[black] (1,2) circle (2pt) node[above] {$0$};
          \filldraw[black] (2,1) circle (2pt) node[above] {$1$};
          \filldraw[black] (2,2) circle (2pt) node[above] {$0$};
          \filldraw[black] (3,1) circle (2pt) node[above] {$3$};
          \filldraw[black] (3,2) circle (2pt) node[above] {$1$};
  
      

      \end{tikzpicture}
      \hfill
      \begin{tikzpicture}[scale=.5]
        \draw[] (-1,2) node {$\!\!\!\!\!\!\!\!\!\!\binom{\a_5+1}{3} = $};
  
        \fill[blue!20, opacity=0.5] (0,0) -- (0,2) -- (1,3) -- (4,3) -- (4,1) -- (3,0) -- cycle;

          \filldraw[black] (1,1) circle (2pt) node[above] {$0$};
          \filldraw[black] (1,2) circle (2pt) node[above] {$0$};
          \filldraw[black] (2,1) circle (2pt) node[above] {$0$};
          \filldraw[black] (2,2) circle (2pt) node[above] {$0$};
          \filldraw[black] (3,1) circle (2pt) node[above] {$1$};
          \filldraw[black] (3,2) circle (2pt) node[above] {$0$};
  

      \end{tikzpicture}
\end{center}
The filtered pieces of the ring $R(\Gamma;\Z)$ are as follows:
  \begin{align*}
    R_0(\Gamma;\Z) & = \Z \{1\}\,,\\
    R_1(\Gamma;\Z) & =\Z \left\{ 1, \a_{C_4}, \a_{C_3} \right\}\,,\\
    R_2(\Gamma;\Z) & = \Z \left\{ 1, \a_{C_4}, \a_{C_3}, \a_{C_4}\a_{C_3}, \binom{\a_{C_4}}{2} \right\}\,,\\
    R_3(\Gamma;\Z) & = R(\Gamma;\Z).
  \end{align*}
Taking the associated graded, we find that the Poincar\'e polynomial of $\IZ(\Gamma; \Z)$ is $1 + 2t + 2t^2 + t^3$.

One interesting feature of this example is that, while the zonotope $Z(\Gamma)$ is not a rectangle, the set $\cZ_-(\Gamma)$ of interior
lattice points is equal to the product $\{1,2,3\}\times\{1,2\}$.  For this reason, the ring $\IZ(\Gamma; \Z)$ factors as the tensor product
of two truncated divided powers algebras of the form that we saw in Example \ref{ex:cycle}, one with $k=4$ and one with $k=3$.
It also means that the relations $$\binom{\a_{C_5}+1}{4}\in \cK_-(\chi^!_\Gamma;\Z)
\and 
\frac{\a_{C_5}^{4}}{4!} \in \cI_-(\chi^!_\Gamma;\Z)$$
turn out to be redundant.

\section{Graphical configuration spaces}\label{sec:config}
As above, let $\Gamma = (V,A,h,t)$ be a directed graph.
This section is devoted to the study of the space $X(SU(2),\Gamma)$ and its cohomology, with the goal of proving Theorem \ref{cohomology}.
Though the space itself does not depend on the orientation of the arrows, we will make use of the orientation in our analysis.

We will use the letter $\kappa$ to denote a function from $V$ to $G$ with $\kappa(h(a))\neq \kappa(t(a))$ for all $a\in A$.
An element of $X(SU(2),\Gamma)$ will then be an equivalence class $[\kappa]$, where $[\kappa_1] = [\kappa_2]$ if and only if $\kappa_1$
and $\kappa_2$ are related by left multiplication by an element of $G^{\pi_0(\Gamma)}$.

\subsection{Cohomology groups}\label{sec:groups}
In this section, we compute the additive structure on the cohomology of $X(SU(2),\Gamma)$.
As in Section \ref{sec:filtrations}, our main tool will be an induction based on the operations of deletion and contraction.
Suppose that $a\in A$ is neither a loop nor a bridge, let $\Gamma'$ be the graph obtained by deleting $a$, and let $\Gamma''$
be the graph obtained by contracting $a$.

\begin{remark}
In Section \ref{sec:filtrations}, we defined deletion and contraction at the level of vector arrangements.
Note that we have $(\chi_\Gamma)' = \chi_{\Gamma'}$ and $(\chi_\Gamma)'' = \chi_{\Gamma''}$,
while $(\chi_\Gamma^!)' = \chi_{\Gamma''}^!$ and $(\chi_\Gamma^!)'' = \chi_{\Gamma'}^!$.  That is, dualization swaps deletion with contraction.
\end{remark}

We begin by observing that we have an open inclusion $X(G,\Gamma)\subset X(G,\Gamma')$, as well as a canonical
homeomorphism $$X(G,\Gamma')\setminus X(G,\Gamma)\cong X(G,\Gamma'').$$
There is a smooth map
$\psi_a:X(G,\Gamma')\to G$ sending $[\kappa]$ to $\kappa( h(a))^{-1}\kappa( t(a))$.  The identity is a regular value of $\psi_a$,
and the fiber over the identity is equal to the closed subset $X(G,\Gamma'')\subset X(G,\Gamma')$.
This implies that the normal bundle to $X(G,\Gamma'')$ in $X(G,\Gamma')$ is trivial.

\begin{proposition}\label{LES}
Let $d$ be the dimension of $G$.
There is a canonical long exact sequence of integral cohomology groups
$$\cdots\to H^i(X(G,\Gamma');\Z) \to H^i(X(G,\Gamma);\Z) \to H^{i-d+1}(X(G,\Gamma'');\Z)
\to H^{i+1}(X(G,\Gamma');\Z) \to\cdots.$$
\end{proposition}

\begin{proof}
By the Tubular Neighborhood Theorem, the Excision Theorem, and the K\"unneth
Theorem, we have
\begin{eqnarray*}H^{i+1}(X(G,\Gamma'), X(G,\Gamma);\Z)&\cong& H^{i-d+1}(X(G,\Gamma'');\Z)\otimes H^{d}(\R^{d}, \R^d\setminus\{0\};\Z)\\
&\cong& H^{i-d+1}(X(G,\Gamma'');\Z).
\end{eqnarray*}
The proposition is obtained by combining this isomorphism with
the long exact sequence of the pair $(X(G,\Gamma'), X(G,\Gamma))$.
\end{proof}

We define the {\bf rank} of $\Gamma$, denoted $\rk(\Gamma)$, to be the number of vertices minus the number of connected components.
Equivalently, it is the dimension of $X(U(1),\Gamma)$, or the rank of the graphical matroid associated with $\Gamma$.
We write $$T_\Gamma(x,y) = T_{\chi_\Gamma}(x,y) = T_{\chi_\Gamma^!}(y,x)$$ to denote the Tutte polynomial of $\Gamma$.
The next result is specific to the case where $G = SU(2)$.

\begin{proposition}\label{additive}
The integral cohomology of $X(SU(2),\Gamma)$ is torsion-free and vanishes in odd degree,
and we have
$$\sum_{i=0}^\infty t^i \dim H^{2i}(X(SU(2),\Gamma);\Q)= t^{\rk(\Gamma)}T_\Gamma(t^{-1},0).$$
\end{proposition}

\begin{proof}
If $\Gamma$ has any loops, then $X(SU(2),\Gamma) = \emptyset$ and $T_\Gamma(x,0) = 0$, so the statement is trivial.
If $\Gamma$ is a forest (that is, if every edge is a bridge), then $X(SU(2),\Gamma) \cong \R^{3\rk(\Gamma)}$ and $T_\Gamma(x,0) = x^{\rk(\Gamma)}$,
so again the statement is true.  Now suppose that there is at least one arrow $a$ that is neither a bridge or a loop, and 
assume by induction that the statement holds for both $\Gamma'$ and $\Gamma''$.

The vanishing of the odd degree cohomology of $X(SU(2),\Gamma)$ is immediate from the inductive hypothesis and Proposition \ref{LES},
with $G = SU(2)$ and $d=3$.
Furthermore, Proposition \ref{LES} gives us the short exact sequence
\begin{equation}\label{SES}
                0\to H^{2i}(X(SU(2),\Gamma');\Z)\to H^{2i}(X(SU(2),\Gamma);\Z)\to H^{2i-2}(X(SU(2),\Gamma'');\Z)\to 0.
\end{equation}
The left and right terms are torsion-free by our inductive hypothesis, thus so is the middle term.  After tensoring with $\Q$,
the dimension of the middle term is equal to the sum of the dimensions of the left and right terms.
The proposition then follows from the identity
$$T_\Gamma(x,y) = T_{\Gamma'}(x,y) + T_{\Gamma''}(x,y)$$
along with the observation that $\rk(\Gamma') = \rk(\Gamma) = \rk(\Gamma'') + 1$.
\end{proof}

\begin{remark}
                In light of Theorem~\ref{cohomology}, the short exact 
                sequence (\ref{SES}) is an integral version of the 
                linear dual of the short exact sequence of 
                \cite[Proposition 4.4]{ArPo} associated with the 
                arrangement of hyperplanes normal to $\chi_\Gamma^!$.
\end{remark}

\begin{remark}
Equation \eqref{IZ-Tutte} and Proposition \ref{additive} say that the ring $\IZ(\Gamma)$ and the space $X(SU(2),\Gamma)$ both have Poincar\'e polynomial equal to 
$$t^{\rk(\Gamma)} T_{\chi_\Gamma^!}(0,t^{-1}) = t^{\rk(\Gamma)} T_{\Gamma}(t^{-1},0).$$  These two results are consistent with the statement of Theorem \ref{cohomology}, and will indeed be used in the proof of the theorem.
\end{remark}

\begin{example}
If $\Gamma$ is the graph from Section \ref{ex:house-graph},
Proposition \ref{additive} says that the Poincar\'e polynomial of $X(SU(2),\Gamma)$ is equal to 
$1 + 2t + 2t^2 + t^3$.
\end{example}

\subsection{Equivariant cohomology}\label{sec:ec}
Let $\T$ be a 1-dimensional torus, and $X$ a $\T$-space whose
integral cohomology vanishes in odd degree.
We will be interested in the equivariant cohomology ring $H^*_\T(X;\Z)$, which is a graded algebra over $H^*_\T(*;\Z) \cong \Z[u]$.

For all $i\geq 0$, let $F_i \subset H^*(X^\T;\Z)$ denote the image of the map
$$H^{2i}_\T(X;\Z)\to H^{2i}_\T(X^\T;\Z) \cong \(H^*(X^\T;\Z)\otimes \Z[u]\)^{2i} \overset{u\mapsto 1}{\longrightarrow} H^*(X^\T;\Z).$$
We clearly have $F_i\subset F_{i+1}$ and $F_i\cdot F_j\subset F_{i+j}$, so we have an increasing filtration
$$F_0\subset F_1 \subset F_2 \subset \cdots.$$
This is called the {\bf equivariant filtration}, and it is functorial:  if $f:X\to Y$ is a $\T$-equivariant map,
then $f^*:H^*(Y^\T;\Z)\to H^*(X^\T;\Z)$ takes the $i^\text{th}$ filtered piece of $H^*(Y^\T;\Z)$ to the $i^\text{th}$ filtered piece of $H^*(X^\T;\Z)$.
We can define the {\bf Rees algebra}
$$\Rees H^*(X^\T;\Z) := \bigoplus_{i=0}^\infty u^i F_i \subset H^*(X^\T;\Z) \otimes \Z[u],$$
which is a graded $\Z[u]$-algebra,
and the {\bf associated graded}
$$\gr H^*(X^\T;\Z) := \bigoplus_{i=0}^\infty F_i/F_{i-1} \cong \Rees H^*(X^\T;\Z) / \langle u \rangle.$$

\begin{remark}\label{double}
Note that the degree of $u$ is equal to 2, thus the summand $F_i/F_{i-1}\subset \gr H^*(X^\T;\Z)$ sits in degree $2i$.
\end{remark}

The following theorem is essentially due to Borel \cite{Borel}, but the integral version appears
in \cite[Proposition 2.8]{DBPW}.

\begin{theorem}\label{Rees}
Suppose that $X$ is a finite dimensional manifold and that $\T$ acts freely on $X\setminus X^\T$.
There exist canonical graded ring isomorphisms
$$H^*_\T(X;\Z) \cong \Rees H^*(X^\T;\Z)\and H^*(X;\Z) \cong \gr H^*(X^\T;\Z).$$
\end{theorem}

\begin{example}\label{cp1}
Let $\T$ act on $X = S^2 = \{(x,y,z)\mid x^2 + y^2 + z^2 = 1\}$ by counterclockwise rotation about the $z$-axis;
the fixed point set consists of the two poles $s = (0,0,-1)$ and $n = (0,0,1)$.
We have $$H^*(X^\T;\Z) = H^0(X^\T;\Z) = \Z\oplus\Z,$$ and $F_0 = \Z\cdot (1,1)$.  We can deduce from Theorem \ref{Rees} that $F_1 = \Z\oplus\Z$,
but let us instead show it directly.

Fix a complex structure (and therefore also an orientation) on $X$ with the property that $\T$ acts on $T_sX$ with weight $-1$
and on $T_nX$ with weight $1$.  Let $\cO(1)$ be the unique complex line bundle with the property that $\cO(1)^{\otimes 2} \cong TX$;
this can also be characterized as the unique line bundle whose first Chern class evaluates to $1$ on the fundamental cycle $[X]$.
There are infinitely many ways to put a $\T$-equivariant structure on $\cO(1)$, any two of which differ by a character of $\T$.
All of them have the property that the weight of the $\T$-action on the fiber $\cO(1)_s$ is one less than the weight of the $\T$-action
on the fiber $\cO(1)_n$.  Choose one of them, with weights $m$ and $m+1$ for some $m\in\Z$.

The $\T$-equivariant Euler class of $\cO(1)$ is an element of $H^2_\T(X;\Z)$.  The restriction of this cohomology class to $X^\T$
is equal to the equivariant Euler class of $\cO(1)|_{X^\T}$.  By our statement about the weights, this class is equal to $(mu,(m+1)u)$.
Setting $u$ equal to 1, we find that the class $(m,m+1)$ is contained in $F_1$.  Since $F_1$ contains both $(1,1)$ and $(m,m+1)$, 
it is equal to $\Z\oplus\Z$.
%
\end{example}

\begin{example}\label{SU2U1}
For any $z\in U(1)$ and $$\begin{pmatrix}a & b\\ c& d\end{pmatrix}\in SU(2),$$
we have
$$\begin{pmatrix}z & 0\\ 0& z^{-1}\end{pmatrix}
\begin{pmatrix}a & b\\ c& d\end{pmatrix}
\begin{pmatrix}z & 0\\ 0& z^{-1}\end{pmatrix}^{-1}
= \begin{pmatrix}a & z^2b\\ z^{-2}c& d\end{pmatrix}.$$
It follows that the conjugation action of $U(1)$ on $SU(2)$ descends to an action of $\T := U(1)/\{\pm 1\}$,
the fixed point set is $U(1)$, and the action is free away from the fixed point set.
This translates to the statement that the action of $\T$ on $X(SU(2),\Gamma)$ by inverse right multiplication has fixed point
set $X(U(1),\Gamma)$, and the action is free away from the fixed point set.
Proposition \ref{additive} tells us that the cohomology of $X(SU(2),\Gamma)$ vanishes in odd degree,
so Theorem \ref{Rees} applies to $X(SU(2),\Gamma)$.
\end{example}

\subsection{Homotopy equivalence}\label{sec:homotopy}
The goal of this section is to construct a canonical homotopy equivalence from $X(U(1),\Gamma)$
to the discrete space $\cZ_-(\Gamma)$.  We begin by introducing a new manifold that is diffeomorphic to $X(G,\Gamma)$
for any Lie group $G$.

Let $Y(G,\Gamma)$ denote the space of maps $\lambda:A\to G\setminus\{\id\}$ with the property that 
$$\lambda(a_1)^{\eps_1}\cdots\lambda(a_k)^{\eps_k} = \id$$
for all oriented cycles $\{(a_1,\eps_1),\ldots,(a_k,\eps_k)\}$.  
Note that this condition is unchanged by cyclic reordering of the arrows.
An element of $Y(G,\Gamma)$ is called a {\bf balanced \boldmath{$G$}-gain graph} with underlying graph $\Gamma$ and all coefficients nontrivial; see for example
\cite{zaslavsky-biased-graphs1,gain}.
We have a diffeomorphism
\begin{equation}\label{diffeo}f:X(G,\Gamma)\overset{\cong}{\longrightarrow} Y(G,\Gamma)\end{equation}
given by putting $f([\kappa])(a) := \kappa( t(a))^{-1}\kappa( h(a))$ for all $a\in A$.
If $G = SU(2)$, we have an action of $\T$ on $Y(SU(2),\Gamma)$ by conjugation, and the diffeomorphism $f$ is $\T$-equivariant.

Let $C_1(\Gamma;\R)$ be the vector space of cellular 1-chains on $\Gamma$, regarded as a cell complex, and let $C^1(\Gamma;\R)$ be the
dual space of cellular 1-cochains.  Since $\Gamma$ has no 2-cells, we have an
injection $$H_1(\Gamma; \R) \overset{\iota}{\longrightarrow} C_1(\Gamma; \R)$$ along with a dual surjection
$$H^1(\Gamma; \R) \overset{\pi}{\longleftarrow} C^1(\Gamma; \R).$$
Given an oriented cycle $C = \{(a_1,\eps_1),\ldots,(a_k,\eps_k)\}\subset A\times\{\pm 1\}$, we have a corresponding homology class 
$[C]\in H_1(\Gamma; \Z)$ characterized by the equation
$$\iota[C] = \sum_{i=1}^k \eps_i a_i\in C_1(\Gamma; \Z) \subset C_1(\Gamma; \R).$$

We have an embedding $$\zeta: Y(U(1),\Gamma)\to C^1(\Gamma;\R)$$
given by the formula $$\zeta(\lambda)(a) := \arg(\la(a)),$$
where $\arg:U(1)\setminus\{\id\}\to (0,1)$ is the normalized argument function (that is, the usual argument function divided by $2\pi$).  For any oriented cycle 
$C$, we have $$\lambda(a_1)^{\eps_1}\cdots\lambda(a_k)^{\eps_k} = \id,$$
and therefore $$\Big\langle \zeta(\lambda), \iota[C]\Big\rangle \in \Z.$$
That is, we have $$\zeta(\la) \in \pi^{-1}H^1(\Gamma;\Z) = \ker(\pi) + C^1(\Gamma; \Z).$$  Since $Y(U(1),\Gamma)$ is cut out by precisely these equations, we may conclude that
\begin{equation}\label{image}\im(\zeta) = (0,1)^A \cap \pi^{-1}H^1(\Gamma;\Z) \subset C^1(\Gamma; \R).\end{equation}

\begin{proposition}\label{homotopy equivalence}
The map $\pi\circ \zeta:Y(U(1),\Gamma)\to \cZ_-(\Gamma)$ is a homotopy equivalence.
That is, each connected component of $Y(U(1),\Gamma)$ is contractible, and $\pi_0(Y(U(1),\Gamma)$ is in bijection with $\cZ_-(\Gamma)$.
\end{proposition}

\begin{proof}
We first argue that $\pi\circ \zeta$ is surjective.  We have
$Z(\Gamma) := \pi([0,1]^A)$ and therefore $$\cZ_-(\Gamma) = \pi((0,1)^A) \cap H^1(\Gamma;\Z) = \pi\((0,1)^A \cap \pi^{-1}H^1(\Gamma;\Z)\),$$
so surjectivity of $\pi\circ \zeta$ follows from Equation \eqref{image}.
For any $q\in \cZ_-(\Gamma)$, 
we can choose $p\in C^1(\Gamma; \Z)$ with $\pi(p) = q$, and then observe that
$$(\pi\circ \zeta)^{-1}(q) \cong \pi^{-1}(q) \cap \im(\zeta) =  (0,1)^A \cap \left(\ker(\pi) + p\right)$$ is convex, and in particular contractible.
\end{proof}

From Equation \eqref{diffeo} and Proposition \ref{homotopy equivalence},
we now have a homotopy equivalence $$\theta := \pi\circ \zeta\circ f:X(U(1),\Gamma)\to \cZ_-(\Gamma),$$ and therefore an isomorphism
$$\theta^*:R(\Gamma;\Z)\to H^0(X(U(1),\Gamma);\Z).$$
%
If we can show that the isomorphism $\theta^*$ takes the filtration of $R(\Gamma;\Z)$ defined in Section \ref{sec:new}
to the equivariant filtration of $H^*(X(U(1),\Gamma);\Z)$, then Theorem \ref{cohomology} will follow from Theorem \ref{Rees}.
The fact that the first isomorphism in Theorem \ref{cohomology} halves degrees is a consequence of the fact that
$R_i(\Gamma;\Z)/R_{i-1}(\Gamma;\Z)\subset \IZ(\Gamma;\Z)$ sits in degree $i$, while $F_i/F_{i-1}\subset H^*(X(SU(2),\Gamma);\Z)$
sits in degree $2i$ (see Remark \ref{double}).
Proving that $\theta^*$ is compatible with the two filtrations will be the goal of the next two sections.

\subsection{The case of a cycle}\label{sec:cycle}
Let $C_k$ be the cycle graph with vertex set $\{1,\ldots,k\}$ and arrow set $\{a_1,\ldots,a_k\}$,
where $a_i$ either points from $i$ to $i+1$ or from $i+1$ to $i$ for all $1\leq i\leq k-1$ and $a_k$ either points from $k$ to $1$ or from $1$ to $k$.
We will assume that $C = \{(a_1,\eps_1),\ldots,(a_k,\eps_k)\}$ is an oriented cycle with $t(\eps_i a_i) = i$ for all $1\leq i\leq k$,
which implies that $h(\eps_i a_i) = i+1$ for $i<k$ and $h(\eps_k a_k) = 1$.
Using the oriented cycle to identify $\Lambda = H^1(C_k;\Z)$ with $\Z$, we have
$$\cZ_-(C_k) =
(-|C^-|,|C^+|)\cap \Z
= \{-|C^-|+1,\cdots, |C^+|-1\}
\;\subset\; \Z = H^1(C_k;\Z),$$
where $C^+$ and $C^-$ are as in Section \ref{sec:cographical}.
In the case where $\eps_i=1$ for all $i$, this coincides with the vector arrangement in Example \ref{ex:cycle}.

Recall the homotopy equivalence $$\theta:X(U(1),C_k)\to \cZ_-(C_k)$$
from Section \ref{sec:homotopy}.  In coordinates, if $\kappa = (\kappa_1,\ldots,\kappa_k)\in U(1)^k$ with $\kappa_i\neq \kappa_{i+1}$ 
for all $1\leq i \leq k-1$ and $\kappa_k\neq \kappa_1$, then 
\begin{eqnarray*}\theta([\kappa]) &=& \eps_1\arg((\kappa_1^{-1}\kappa_2)^{\eps_1}) + \eps_2\arg((\kappa_2^{-1}\kappa_3)^{\eps_2})
+ \cdots + \eps_{k-1}\arg((\kappa_{k-1}^{-1}\kappa_k)^{\eps_{k-1}}) + \eps_k\arg((\kappa_k^{-1}\kappa_1)^{\eps_k})\\
&=& -|C^-| + \arg(\kappa_1^{-1}\kappa_2) + \arg(\kappa_2^{-1}\kappa_3)
+ \cdots + \arg(\kappa_{k-1}^{-1}\kappa_k) + \arg(\kappa_k^{-1}\kappa_1).\end{eqnarray*}
That is, we walk from $\kappa_1$ to $\kappa_2$
to $\kappa_3$ and so on, eventually coming back to $\kappa_1$.  For the $i^\text{th}$ step, we walk counter-clockwise if $\eps_i=1$
and clockwise if $\eps_i=-1$.  In the end, $\theta([\kappa])$ records the winding number of our path, which lies strictly between $-|C^-|$ and $|C^+|$.

In this section, we will show that $F_i =  \theta^* R_i(C_k;\Z) \subset R(C_k;\Z)$, where $F_i$ is the $i^\text{th}$
piece of the equivariant filtration of $H^0(X(U(1),C_k);\Z)$.
Let $\eta\in R_1(C_k;\Z)$ be the identity function; that is, $\eta(m) := m$ for all $m\in\cZ_-(C_k)$.  
The proof of following lemma is the most technical part of the paper.

\begin{lemma}\label{heart}
We have $\theta^*(\eta)\in F_1 \subset H^0(X(U(1),C_k);\Z)$.
\end{lemma}

\begin{proof}
Consider the arrow $a_k$ of $C_k$.  
The graph obtained by deleting this arrow is a path $P_{k-1}$ of length $k-1$, and the graph obtained by contracting this
arrow can be identified with the $(k-1)$-cycle $C_{k-1}$.
In particular, we have an open embedding of $X(SU(2),C_k)$ into 
$$X(SU(2),P_{k-1}) \cong (SU(2)\setminus\{\id\})^{k-1}\cong \R^{3(k-1)},$$
and the complement is identified with $X(SU(2),C_{k-1})$.
Moreover, we have a smooth map\footnote{The map $\psi$ is equal to the map $\psi_{a_k}$ from Section \ref{additive} if $\eps_k=1$,
and to the composition of this map with the inverse function if $\eps_k=-1$.}
\begin{align*}\psi:X(SU(2),P_{k-1})&\to SU(2)\\
[(\kappa_1,\ldots,\kappa_k)] &\mapsto \kappa_1^{-1}\kappa_k
\end{align*}
with the property that the identity is a regular value, the fiber over the identity
is equal to $X(SU(2),C_{k-1})$, and the preimage of $SU(2)\setminus\{\id\}$ is equal to $X(SU(2),C_k)$.

Proposition \ref{LES} provides us with an isomorphism
$$H^2(X(SU(2),C_k);\Z)\cong H^0(X(SU(2),C_{k-1});\Z),$$
and we will write $e\in H^2(X(SU(2),C_k);\Z)$ to denote the element of the left-hand side that is identified with the
canonical generator of the right-hand side.
Informally, $e$ is characterized by the property that it takes the value 1 on the homology class of a small 2-sphere that links
with $X(SU(2),C_{k-1})\subset X(SU(2),P_{k-1})$.

We now give a more precise characterization of the class $e$.
Recall from Example \ref{cp1} that we have fixed a complex structure, orientation, $\T$-action, and line bundle $\cO(1)$ on 
the 2-sphere $S^2$,
and that we write $s=(0,0-1)$ and $n=(0,0,1)$ for the two poles.
For any small number $\varepsilon>0$, consider the 
homeomorphism from $S^2$ to the boundary of a small ball around $\id\in SU(2)$
given by the formula $$\iota_\varepsilon(x,y,z) := \exp\left(2\pi \varepsilon\begin{pmatrix}zi & x+yi\\ -x+yi & -zi \end{pmatrix}\right)
= \cos(2\pi \varepsilon)\cdot \id + \sin(2\pi \varepsilon) \cdot \begin{pmatrix}zi & x+yi\\ -x+yi & -zi \end{pmatrix}.$$
Note in particular that
$$\iota_\varepsilon(s) = \begin{pmatrix}e^{-2\pi i\varepsilon} & 0\\ 0 & e^{2\pi i\varepsilon} \end{pmatrix}
\and
\iota_\varepsilon(n) = \begin{pmatrix}e^{2\pi i\varepsilon} & 0\\ 0 & e^{-2\pi i\varepsilon} \end{pmatrix}.$$
Now suppose that we have an inclusion $\xi:S^2\to X(SU(2),C_k)$ such that $\psi\circ\xi=\iota_\varepsilon$.
The class $e\in H^2(X(SU(2),C_k);\Z)$ is characterized by the property that
$\big\langle e, [\xi]\big\rangle = 1$.  Equivalently, the complex line bundle $L$ on $X(SU(2),C_k)$ with Euler class $e$
is characterized by the property that $\xi^*L$ is isomorphic to $\cO(1)$. 

By Theorem \ref{Rees}, $e$ lifts to a $\T$-equivariant cohomology class, thus $L$ admits a $\T$-equivariant structure.
There are infinitely many such choices, any two of which differ by a character of $\T$.  We will chose the one that is uniquely
characterized by the property that $\T$ acts with weight $-|C^-|+1$ on the fiber $L_{[\kappa]}$
for all ${[\kappa]}$ in the connected component $\theta^{-1}(-|C^-|+1)\subset X(U(1),C_k)$.

Let $\rho\in F_1\subset H^0(X(U(1),C_k);\Z)$ be the class obtained by restricting the $\T$-equivariant Euler class of $L$ to $X(U(1),C_k)$
and then setting $u=1$.  In concrete terms, the value of $\rho$ at a point $[\kappa]\in X(U(1),C_k)$ is equal to the weight of the $\T$-action on the fiber $L_{[\kappa]}$.  In particular, $\rho([\kappa])=-|C^-|+1$ for all $[\kappa]\in \theta^{-1}(-|C^-|+1)$.  We will prove by induction that 
$\rho([\kappa]) = m$ for all $[\kappa]\in \theta^{-1}(m)$, and therefore that 
$\rho=\theta = \theta^*(\eta)$.  This will complete the proof of the 
lemma.

\excise{
For any positive real number $r$, consider the $\T$-equivariant map
$\iota_r:S^2\to SU(2)\setminus\{\pm\id\}$ given by the formula
$$\iota_r(x,y,z) := \frac{1}{r^2+1}\begin{pmatrix}r^2-1+2zi & 2x+2yi\\ -2x+2yi & r^2-1-2zi \end{pmatrix}.$$
Note that $\iota_1$ coincides with the map $\iota$ defined above, and therefore $\iota_r$ is homotopic to $\iota$ for all $r$.   
By taking $r$ to be sufficiently large, we may ensure that the image of
$\iota_r$ is arbitrarily close to the element $\id\in SU(2)$.  
In particular,
$$\iota_r(n) = \iota_r(0,0,1) = \begin{pmatrix}\frac{r^2-1+2i}{r^2+1} & 0\\ 0 & \frac{r^2-1-2i}{r^2+1} \end{pmatrix}$$
is an element of $U(1)\setminus\{\id\}$ with normalized argument very close to $0$, while
$$\iota_r(s) = \iota_r(0,0,-1) = \begin{pmatrix}\frac{r^2-1-2i}{r^2+1} & 0\\ 0 & \frac{r^2-1+2i}{r^2+1} \end{pmatrix}$$
is an element of $U(1)\setminus\{\id\}$ with normalized argument very close to $1$.
}

For our induction, we assume that $\rho([\kappa]) = -|C^-| + r$ for all $[\kappa]\in\theta^{-1}(-|C^-| + r)$
some particular $r<k-1$, and we will prove that the same property holds for $r+1$.
Choose a positive number $\varepsilon < \frac{r}{k-1}$
and consider the $\T$-equivariant inclusion $\xi_\varepsilon:S^2\to X(SU(2),C_k)$ given by the formula
$$\xi_\varepsilon(x,y,z) := \left[\left(1,e^{2\pi r i/(k-1)},e^{4\pi r i/(k-1)},\ldots,e^{2(k-2)\pi r i/(k-1)},\iota_\varepsilon(x,y,z)\right)\right].$$
Our choice of $\varepsilon$ guarantees that $\xi_\varepsilon$ is well-defined, 
and that we have $\theta(\xi_\varepsilon(s)) = -|C^-| + r$ and $\theta(\xi_\varepsilon(n)) = -|C^-| + r+1$.

We have $$\psi\circ\xi_\varepsilon = \iota_\varepsilon,$$
which implies that $\xi_\varepsilon^*L$ is (nonequivariantly) 
isomorphic to $\cO(1)$.  By Example \ref{cp1}, the $\T$-weight $\rho(\xi_\varepsilon(s))$ of the fiber of $\xi_\varepsilon^*L$ at $s$ 
is one less than the $\T$-weight $\rho(\xi_\varepsilon(n))$ of the fiber at $n$.  
Our inductive hypothesis says that the former is equal to $-|C^-| + r$, therefore the latter is equal to $-|C^-| + r+1$.  This completes the induction.
\end{proof}

\begin{proposition}\label{cycle case}
For all $i$, we have $F_i =  \theta^* R_i(C_k;\Z)$.
\end{proposition}

\begin{proof}
We first prove the proposition after tensoring with $\Q$.
We know that $R_1(C_k;\Q)$ is spanned by $1$ and $\eta$, and 
Lemma \ref{heart} tells us that $\theta^* R_1(C_k;\Q)\subset F_1\otimes\Q$.
Since every element of $R_i(C_k;\Q)$ can be written as a polynomial expression of degree $\leq i$ in the elements of $R_1(C_k;\Q)$,
this implies that $\theta^* R_i(C_k;\Q)\subset F_i\otimes\Q$.

By Proposition \ref{Q-version} and Equation \eqref{IZ-Tutte}, we have
$$\sum_{i\geq 0}t^i\dim R_i(C_k;\Q)/R_{i-1}(C_k;\Q) = t^{k-1}T_{\chi_{C_k}^!}(0,t^{-1}) = 1 + t + \cdots + t^{k-2}.$$
Similarly, Proposition \ref{additive} tells us that
$$\sum_{i\geq 0}t^i\dim (F_i\otimes \Q)/(F_{i-1}\otimes\Q) = t^{k-1}T_{C_k}(t^{-1},0) = 1 + t + \cdots + t^{k-2}.$$
Thus the inclusions $\theta^* R_i(C_k,\Q)\subset F_i\otimes\Q$ must be equalities for all $i$.

Now we are ready to prove the statement over the integers.  By our analysis over the rationals, we know that $F_i$ is a finite index
sublattice of $\theta^* R_i(C_k;\Z)$.  Proposition \ref{additive} and Theorem \ref{Rees} tell us that $F_i/F_{i-1}$ is torsion-free for all $i$, which implies that
$F_i$ is saturated for all $i$.  Thus the inclusion of $F_i$ into $\theta^* R_i(C_k;\Z)$ must be an equality.
\end{proof}

\subsection{The general case}\label{sec:proof}
We are now ready to prove Theorem \ref{cohomology} for arbitrary graphs.  As explained at the end of Section \ref{sec:homotopy}, it is sufficient to prove
that $F_i = \theta^*R_i(\Gamma;\Z)$ for all $i$.

\begin{proof}[Proof of Theorem \ref{cohomology}]
Our first task is to formulate and prove the analogue of Lemma \ref{heart}.
The group $R_1(\Lambda,\Z)\subset R(\Lambda,\Z)$ is spanned by the unit $1$ and the elements
$\{\a_C\mid \text{$C$ an oriented cycle}\}$ defined in Section \ref{sec:cographical}.
Let $\eta_C\in R_1(\Gamma;\Z)$ be the restriction of $\a_C$ to $\cZ_-(\Gamma)$.
Proposition \ref{saturation} tells us that $R_1(\Gamma;\Z)$ is spanned by the unit 1 and the elements 
$\{\eta_C\mid \text{$C$ an oriented cycle}\}$.
The analogue of Lemma \ref{heart} is the statement that, for all $C$, \begin{equation}\label{analogue}\theta^*\eta_C\in F_1\subset H^0(X(U(1),\Gamma);\Z).\end{equation}

Let $k = |C|$.  The oriented cycle $C$ defines an inclusion of $C_k$ into $\Gamma$.  This induces restriction
maps $\mu:X(U(1),\Gamma)\to X(U(1),C_k)$ and $\nu:\cZ_-(\Gamma)\to \cZ_-(C_k)$, which in turn induce pullback
maps $\mu^*:H^0(X(U(1),C_k); \Z)\to H^0(X(U(1),\Gamma); \Z)$ and $\nu^*:R(C_k,\Z)\to R(\Gamma;\Z)$.
Since $\mu$ is the restriction to the fixed point set of a $\T$-equivariant map $X(SU(2),\Gamma)\to X(SU(2),C_k)$, the map $\mu^*$ is compatible with the equivariant filtrations.

Consider the commutative diagram
\[
\begin{tikzcd}
R(C_k;\Z) \ar[rr, "\theta_{C_k}^*"]\ar[dd, "\nu^*"] && H^0(X(U(1),C_k);\Z)\ar[dd, "\mu^*"]\\ \\
R(\Gamma;\Z) \ar[rr, "\theta^*"] && H^0(X(U(1),\Gamma);\Z).
\end{tikzcd}
\]
We have $$\theta^*(\eta_C) = \theta_\Gamma^*(\nu^*\eta) = \mu^*(\theta_{C_k}^*\eta).$$
Lemma \ref{heart} tells us that $\theta_{C_k}^*\eta$ is in the first filtered piece of $H^0(U(1),C_k);\Z)$.
Since the map $\mu^*$ respects filtrations, this implies that $\mu^*(\theta_{C_k}^*\eta)$ is in the first filtered piece of $H^0(U(1),\Gamma);\Z)$.
This completes the proof of Equation \eqref{analogue}.

The remainder of the proof is basically the same as the proof of Proposition \ref{cycle case}; we will briefly summarize it here.
The first step is to prove that $F_i \otimes \Q = \theta^* R_i(\Gamma;\Q)$.
We know that $R_1(\Gamma;\Q)$ is spanned by the classes $\eta_C$, thus Equation \eqref{analogue} implies that $\theta^* R_1(\Gamma;\Q) \subset F_1\otimes\Q$,
and therefore that $\theta^* R_i(\Gamma;\Q) \subset F_i\otimes\Q$.
By Proposition \ref{Q-version}, Equation \eqref{IZ-Tutte}, and Proposition \ref{additive}, we have
$$\sum_{i\geq 0}t^i\dim R_i(\Gamma;\Q)/R_{i-1}(\Gamma;\Q) = t^{\rk(\Gamma)-1}T_{\chi_{\Gamma}^!}(0,t^{-1}) = t^{\rk(\Gamma)-1} T_{\Gamma}(t^{-1},0)
= \sum_{i\geq 0}t^i\dim (F_i\otimes \Q)/(F_{i-1}\otimes\Q),$$ thus the inclusion of $\theta^* R_i(\Gamma;\Q)$ into $F_i\otimes\Q$ must be an equality for all $i$.
Having established the theorem over the rationals, we can now conclude that $F_i$ is a finite index sublattice of $\theta^*R_i(\Gamma;\Z)$.
Proposition \ref{additive} and Theorem \ref{Rees} tell us that $F_i/F_{i-1}$ is torsion-free for all $i$, therefore the inclusion of $F_i$ into $\theta^* R_i(\Gamma;\Z)$ must also be an equality.
\end{proof}

\bibliographystyle{amsalpha}
\bibliography{symplectic}

\newcommand{\etalchar}[1]{$^{#1}$}
\def\cprime{$'$}
\providecommand{\bysame}{\leavevmode\hbox to3em{\hrulefill}\thinspace}
\providecommand{\MR}{\relax\ifhmode\unskip\space\fi MR }
\providecommand{\MRhref}[2]{%
  \href{http://www.ams.org/mathscinet-getitem?mr=#1}{#2}
}
\providecommand{\href}[2]{#2}
\begin{thebibliography}{DBPW24}

\bibitem[AP10]{ArPo}
Federico Ardila and Alexander Postnikov, \emph{Combinatorics and geometry of power ideals}, Trans. Amer. Math. Soc. \textbf{362} (2010), no.~8, 4357--4384.

\bibitem[AP15]{ArPo-correction}
\bysame, \emph{Correction to ``{C}ombinatorics and geometry of power ideals'': two counterexamples for power ideals of hyperplane arrangements [mr2608410]}, Trans. Amer. Math. Soc. \textbf{367} (2015), no.~5, 3759--3762.

\bibitem[Ber18]{Berget-reflection}
Andrew Berget, \emph{Internal zonotopal algebras and the monomial reflection groups {$G(m,1,n)$}}, J. Combin. Theory Ser. A \textbf{159} (2018), 1--25.

\bibitem[BLP{\etalchar{+}}11]{kosdef}
Tom Braden, Anthony Licata, Christopher Phan, Nicholas Proudfoot, and Ben Webster, \emph{Localization algebras and deformations of {K}oszul algebras}, Selecta Math. (N.S.) \textbf{17} (2011), no.~3, 533--572.

\bibitem[Bor60]{Borel}
Armand Borel, \emph{Seminar on transformation groups}, Annals of Mathematics Studies, No. 46, Princeton University Press, Princeton, N.J., 1960, With contributions by G. Bredon, E. E. Floyd, D. Montgomery, R. Palais.

\bibitem[Bro24]{Brodsky}
Sarah~B. Brodsky, \emph{The zonotopal algebra of the broken wheel graph and its generalization}, Electron. J. Combin. \textbf{31} (2024), no.~1, Paper No. 1.49, 35.

\bibitem[CDD21]{gain}
Matteo Cavaleri, Daniele D'Angeli, and Alfredo Donno, \emph{A group representation approach to balance of gain graphs}, J. Algebraic Combin. \textbf{54} (2021), no.~1, 265--293.

\bibitem[CMR24]{CMR}
Raymond Chou, Tomoo Matsumura, and Brendon Rhoades, \emph{Equivariant cohomology of {Grassmannian} spanning lines}, Preprint, {arXiv}:2410.02105 [math.{CO}] (2024), 2024.

\bibitem[CP]{ZA2}
Colin Crowley and Nicholas Proudfoot, \emph{{The geometry of zonotopal algebras II: Orlik--Terao algebras and Schubert varieties}}, \textsf{arXiv:2505.05324}.

\bibitem[DBPW24]{DBPW}
Galen Dorpalen-Barry, Nicholas Proudfoot, and Jidong Wang, \emph{Equivariant cohomology and conditional oriented matroids}, Int. Math. Res. Not. IMRN (2024), no.~11, 9292--9322.

\bibitem[GLW24]{GLW}
Sean~T. Griffin, Jake Levinson, and Alexander Woo, \emph{Springer fibers and the delta conjecture at $t=0$}, Adv. Math. \textbf{439} (2024), 53 (English), Id/No 109491.

\bibitem[GM10]{GM}
Mark Goresky and Robert MacPherson, \emph{On the spectrum of the equivariant cohomology ring}, Canad. J. Math. \textbf{62} (2010), no.~2, 262--283.

\bibitem[HHH05]{HHH}
Megumi Harada, Andr\'e{} Henriques, and Tara~S. Holm, \emph{Computation of generalized equivariant cohomologies of {K}ac-{M}oody flag varieties}, Adv. Math. \textbf{197} (2005), no.~1, 198--221.

\bibitem[HR11]{Zonotopal}
Olga Holtz and Amos Ron, \emph{Zonotopal algebra}, Adv. Math. \textbf{227} (2011), no.~2, 847--894.

\bibitem[KN17]{Kirillov-Nenashev}
Anatol~N. Kirillov and Gleb Nenashev, \emph{On {$Q$}-deformations of {P}ostnikov-{S}hapiro algebras}, S\'em. Lothar. Combin. \textbf{78B} (2017), Art. 55, 12.

\bibitem[Len14]{Lenz}
Matthias Lenz, \emph{Interpolation, box splines, and lattice points in zonotopes}, Int. Math. Res. Not. IMRN (2014), no.~20, 5697--5712.

\bibitem[MPY17]{MPY}
Daniel Moseley, Nicholas Proudfoot, and Ben Young, \emph{The {O}rlik-{T}erao algebra and the cohomology of configuration space}, Exp. Math. \textbf{26} (2017), no.~3, 373--380.

\bibitem[Pag23]{Pagaria}
Roberto Pagaria, \emph{The {F}robenius characteristic of the {O}rlik-{T}erao algebra of type {A}}, Int. Math. Res. Not. IMRN (2023), no.~13, 11577--11591.

\bibitem[PR19]{Pawlowski-Rhoades}
Brendan Pawlowski and Brendon Rhoades, \emph{A flag variety for the delta conjecture}, Trans. Am. Math. Soc. \textbf{372} (2019), no.~11, 8195--8248 (English).

\bibitem[RRT23]{RRT}
Markus Reineke, Brendon Rhoades, and Vasu Tewari, \emph{Zonotopal algebras, orbit harmonics, and {D}onaldson-{T}homas invariants of symmetric quivers}, Int. Math. Res. Not. IMRN (2023), no.~23, 20169--20210.

\bibitem[Sch98]{schrijver}
Alexander Schrijver, \emph{Theory of linear and integer programming.}, repr. ed., Chichester: Wiley, 1998.

\bibitem[Tut65]{tutte}
W.~T. Tutte, \emph{Lectures on matroids}, J. Res. Natl. Bur. Stand., Sect. B \textbf{69} (1965), 1--47.

\bibitem[Yau10]{Yau}
Donald Yau, \emph{Lambda-rings}, World Scientific Publishing Co. Pte. Ltd., Hackensack, NJ, 2010.

\bibitem[Zas89]{zaslavsky-biased-graphs1}
Thomas Zaslavsky, \emph{Biased graphs. {I}. {B}ias, balance, and gains}, J. Combin. Theory Ser. B \textbf{47} (1989), no.~1, 32--52.

\bibitem[Zie95]{Z}
G{\"u}nter~M. Ziegler, \emph{Lectures on polytopes}, Graduate Texts in Mathematics, vol. 152, Springer-Verlag, New York, 1995.

\end{thebibliography}

\end{document}